\documentclass[10pt]{amsart}
\usepackage{a4wide}
\usepackage{amsfonts}

\usepackage{amsmath}
\usepackage{amssymb}
\usepackage{amsthm}
\usepackage{graphicx}

\newtheorem{proposition}{Proposition}

\newtheorem {lemma}{Lemma}
\newtheorem {thm}{Theorem}
\newtheorem {rem}{Remark}

\newcommand{\eps}{\varepsilon}

\newtheorem {corollary}{Corollary}
\newtheorem*{assumption}{Assumption}
\newcommand{\E}{\mathbb{E}}
\newcommand{\N}{\mathbb{N}}
\newcommand{\R}{\mathbb{R}}

\newcommand{\1}{\mathbf{1}}
\renewcommand{\P}{\mathbb{P}}

\newtheorem {example}{Example}
\numberwithin{equation}{section}
\numberwithin{equation}{section}
\numberwithin{equation}{section}

\begin{document}
\title{(Non)Differentiability and Asymptotics for Potential Densities of Subordinators}
\author{Leif D\"oring}

\address{Institut f\"ur Mathematik, Technische Universit\"at Berlin, Stra\ss e des 17.~Juni 136, 10623 Berlin, Germany}
\thanks{The first author was supported by the EPSRC grant EP/E010989/1.}
\email{leif.doering@googemail.com}
\author{Mladen Savov}
\address{New College, University of Oxford, Holywell Street, Oxford OX1 3BN, United Kingdom}
\email{savov@stats.ox.ac.uk}
\date{}

\maketitle
\begin{abstract}
	For subordinators $X_t$ with positive drift we extend recent results on the structure of the potential measures
	\begin{align*}
		U^{(q)}(dx)=\E\left[\int_0^{\infty}e^{-qt}\mathbf{1}_{\{X_t\in dx\}}\,dt\right]
	\end{align*}	
	and the renewal densities $u^{(q)}$.  Applying Fourier analysis a new representation of the potential densities is derived from which we deduce asymptotic results and show how the atoms of the L\'evy measure $\Pi$ translate into points of 	(non)smoothness of $u^{(q)}$.
\end{abstract}

\section{Introduction}
    Let $X=(X_{t})_{t\geq0}$ be a L\'evy process, i.e. a real-valued stochastic process with stationary and independent increments which possesses almost surely right-continuous sample paths and starts from zero. If $X$ has almost surely non-decreasing paths it is called a subordinator and can be thought of as an extension of the renewal process which accommodates for infinitely many jumps on a finite time horizon. In fact any such process can be written as $$X_{t}=\delta t+\sum_{s\leq t}\Delta_{s},$$ where $\Delta_{.}$ are random, positive jumps drawn from a Poisson random measure. The non-negative coefficient $\delta$ is referred to as drift of the subordinator.
    The jumps $\Delta_\cdot$ are described by a deterministic intensity measure $\Pi(dx)$, usually referred to as L\'evy measure, that is concentrated on the non-negative reals and satisfies the integrability condition
	\begin{align}\label{c}
		\int_{0}^{\infty}(1\wedge x)\Pi(dx)<\infty.
	\end{align}
	The main object of our study are the $q$-potential measures, $q\geq 0$,
	\begin{align}\label{Mladen1}
		U^{(q)}(dx)=\E\left[\int_0^{\infty}e^{-qt}\mathbf{1}_{\{X_t\in dx\}}\,dt\right]
	\end{align}
	of $X$ whose distribution function will be denoted by $U^{(q)}(x)=U^{(q)}([0,x])$. The slight abuse of notation will always become clear from the context.
	Being the central object of the potential theory for subordinators, $U^{(q)}$ is a well-studied object. For more information we refer the reader to Chapter 3 of \cite{B96}. Whenever $\delta>0$ each point $x>0$ is visited by $X$ with positive probability and it is known that in this case the measure $U^{(q)}$ admits a bounded and continuous density $u^{(q)}=(U^{(q)})'$ with respect to the Lebesgue measure.

\smallskip
Our main objective is to study the regularity of the $q$-potential densities $u^{(q)}$ in more depth. Under mild restrictions on the L\'evy measure the main results concern differentiability properties of $u^{(q)}$ when $\Pi$ has atoms that may accumulate at zero. We characterize the set of values where differentiability of $u^{(q)}$ fails as well as the order of derivatives where smoothness breaks down for each point of this set. More precisely, if $G$ denotes the set of atoms of $\Pi(dx)$ then the $n$th derivative of $u^{(q)}$ exists precisely outside of $G_{n}$, where $G_{n}$ contains all numbers that can be represented as a sum of at most $n$ elements of $G$. Also we provide asymptotic results for $u^{(q)}(x)$ as $x$ goes to zero and infinity. For the results at infinity we need to restrict to renewal densities for which from now on we use the conventions $u=u^{(0)}$ and $U=U^{(0)}$.

\smallskip
The paper is organized as follows:  Motivation and further background of our research is presented in Section \ref{M1}. An example for the simplest L\'evy measure $\Pi=\delta_1$ is included for which direct computations can be performed to present our abstract results in a simple setting.
 In Section \ref{S2} the main results are presented and are proved in Section \ref{S3}.

\section{Motivation and Background}	\label{M1}
\subsection{Background}\label{MM1}
The $q$-potential measures $U^{(q)}$ are the fundamental quantities in potential theory and underly the study of Markov Processes in general and L\'evy processes in particular. As subordinators appear naturally in the study of Markov processes (see Chapter IV in \cite{B96}) and are in some sense the simplest L\'evy processes, it is not surprising that important information about its potential measures has already appeared in the literature. For our special case when we study a subordinator with positive drift, it is known that with $T_{x}=\inf\{t\geq 0:X_{t}>x\}$
\begin{align}
		u(x)=\frac{1}{\delta}\P[X_{T_x}=x],\label{creep}
\end{align}
i.e. $u(x)$ has the clear probabilistic interpretation as creeping probability of $X$ at level $x$. From this it follows from a renewal argument that $u$ solves the renewal type equation
	$$\delta u(x)+\int_0^xu(x-y)\bar \Pi(y)\,dy=1$$
which will be the starting point for our analysis. This simple renewal structure motivates the alternative naming of $u$ as renewal density. Here and in the following
$$\bar\Pi(y)=\Pi\big([y,\infty)\big)=\int_y^\infty \Pi(dx)$$
denotes the tail of the L\'evy measure. In fact, similar renewal type equations for $u^{(q)}$ are derived below allowing for a study for all $q\geq 0$. 
	
\smallskip
The functions $u^{(q)}$ are rarely known explicitly but can be described using the universal representation of the Laplace transform of $U^{(q)}(dx)$. On page 74 of \cite{B96} it is shown that
\begin{equation}\label{LaplaceUq}
	\int_{0}^{\infty}e^{-\lambda x}U(dx)=\frac{1}{\psi(\lambda)},\,\,\,\lambda>0,
\end{equation}
where $\psi(\lambda)$ is the the Laplace exponent
\begin{align*}
	\psi(\lambda)=\delta \lambda+\int_0^\infty (1-e^{-\lambda x})\Pi(dx).
\end{align*}
In few exceptional cases such as stable subordinators without drift, i.e. $\Pi(dx)=cx^{-1-\alpha}dx$ for some $\alpha\in (1,2)$, this formula sufficiently simplifies and allows for an explicit Laplace inversion.

\smallskip
Smoothness issues for the particular case $q=0$ have attracted some attention in special cases motivated by applications; we shall only highlight two recent results:\\
First, it is known that $u$ is completely monotone and henceforth infinitely differentiable for the class of so-called special subordinators whose Laplace exponent is a special Bernstein function (see \cite{SV06}). This, as we will show below, is not true for general subordinators. Furthermore, for special subordinators, $u$ is a decreasing function. Monotonicity fails in general but we show that a similar structural property holds: For all $q\geq 0$, $u^{(q)}$ is of bounded variation and hence a difference of two increasing functions.\\
Secondly, when the L\'evy measure $\Pi(dx)$ has no atoms and $\delta>0$ it has been shown in \cite{CKS10} that $u$ is continuously differentiable. This result can also be recovered by Kingmann's study of $p$-functions (see Chapter 3 of \cite{K72}) or by our representation \eqref{a} below.

\smallskip
We note that the renewal density $u$ is in fact the Kingmann's $p$-function which is naturally associated to a regenerative event, see \cite{K64} and Chapter 3 in \cite{K72}. Using this link one could recover some of our results below for the special case $q=0$.
 However these results are preliminary in our study and we focus on finer properties of the functions $u^{(q)}$ for $q\geq 0$.

\medskip
A further topic of interest are the asymptotics of $u$ and its derivatives. The asymptotic at zero and infinity are classically deduced from the renewal theorem. The relation
\begin{align}
	 	u(x)&\to \frac{1}{\delta},\quad\text{ as }x\to 0,\label{d}
\end{align}
is stated in Theorem 5 in \cite{B96}, whereas
 \begin{align}
		u(x)&\to \frac{1}{\E[X_1]},\quad\text{ as }x\to \infty,\label{dd}
\end{align}
 seems to have been obtained first in \cite{HS01}. We apply two representations developed in the present paper to understand the asymptotic behaviour of $u^{(q)}$ and $(u^{(q)})'$ in more detail. It turns out that different representations are of very different use: A series representation can be used to deduce the asymptotic properties at zero, whereas a Fourier inversion representation is useful to understand the more delicate asymptotic at infinity. The use of the Riemann-Lebesque theorem for the Fourier inversion representation forces us to assume $q=0$ so that those results are restricted to renewal densities.

\subsection{Applications}\label{MM2}
The motivation for our work is mostly theoretical but we outline below some applications. Our theoretical motivation stems from the mere importance of potential measures in the study of L\'evy processes. As such those appear in various one-sided and double-sided exit problems and especially the probability of hitting a boundary point of an interval at first exit out of it is represented in terms of $u$ or if the L\'evy process is killed at an independent exponential time in terms of $u^{(q)}$ (see Theorem 19 of Chapter VI in \cite{B96}).

\smallskip
Potential densities appear equally in other theoretical studies; For example, representation \eqref{a} below was used to guess the necessary and sufficient analytic condition for existence of right-inverses of L\'evy processes which then, in \cite{DS10}, have been proved via subsequent discoveries.


\smallskip
Our results should be applicable for simulations of subordinators as detailed knowledge of the set of non-smoothness of $u^{(q)}$ could significantly speed up the numerics behind the simulation/computations of potential densities. This has been noticed in relation with scale functions by Kuznetsov, Kyprianou and Rivero (private communication) caused by the celebrated identity $W'(x)=u(x)$ between scale functions of a spectrally negative L\'evy process and the down-going ladder height subordinator (for more information see Chapter 7 of \cite{B96}).  The same holds for potential measures of subordinators. Recall that the scale function is positive and increasing, and is involved in computing double exit probabilities for spectrally negative L\'evy processes $X$, i.e. processes that can jump only downwards. More precisely, with $T_{[-a,b]}=\inf\{t\geq 0:\,X_{t}>b\,\text{ or }X_{t}<-a\}$
\[P\big(X_{T[-y,x]}>x\big)=\frac{W(y)}{W(x+y)}, \text{ for all $x>0$ and $y>0$}.\]

\smallskip
Finally, let us point out a possible extension. For atomic L\'evy measures our main results, Theorem \ref{t4} and Theorem \ref{t9}, link the atoms of $\Pi$ to points of non-differentiability of $u^{(q)}$. It seems possible to extend this to higher order singularities of $\bar\Pi(x)$. Such a generalization would be extremely useful for numerical computations of the scale function  which is a basic quantity in insurance mathematics as it allows for estimation and calculation the double exit probabilities which in the actuarial mathematics context are ruin probabilities. For more background see \cite{CY05} and \cite{AKP04}.
\subsection{Example}\label{MM3}	
	To motivate our findings, here is an illustrative example for which the renewal density $u$ can be calculated explicitly and some non-trivial connections between the atom of $\Pi$ and differentiability of $u$ can be observed.
\begin{example}\label{example}
	Suppose $\delta=1$ and $\Pi$ consists of just one atom at $1$ with mass $1$, i.e. $\Pi=\delta_1$, implying that $\bar \Pi=\mathbf{1}_{[0,1]}$. For $x<1$ one immediately obtains	
	\begin{align*}
		u(x)=\P[X_{T_x}=x]=\P[\text{no jump before }x]=e^{-x\bar\Pi(0)}=e^{-x}.
	\end{align*}
 	For $x\in [n,n+1)$ we exploit the structure of the process more carefully. Since $x<n+1$, on the event of creeping at $x$ there can be at most $n$ jumps of height one before exceeding $x$. If precisely $i\leq n$ jumps occur, they need to appear before $x-i$ as otherwise the drift would have pushed $X$ above level $x$ before the final jump implying
	\begin{align*}
		\P[X_{T_x}=x]&=\sum_{i=0}^{n}\P[\text{exactly } i \text{ jumps before }x-i].
	\end{align*}
	As the number of jumps before $x-i$ is Poissonian with parameter $x-i$ we obtain
	\begin{align*}
		u(x)=
		\begin{cases}
			e^{-x}&:x<1,\\
			e^{-x}+\sum_{i=1}^n\frac{(x-i)^i}{i!}e^{-(x-i)}&:x\in[n,n+1).
		\end{cases}
	\end{align*}	
	The graphs of $u$ and it's derivative are plotted in Figure \ref{fig:u}.
	\begin{figure}
			\begin{center}
		\includegraphics[scale=0.2]{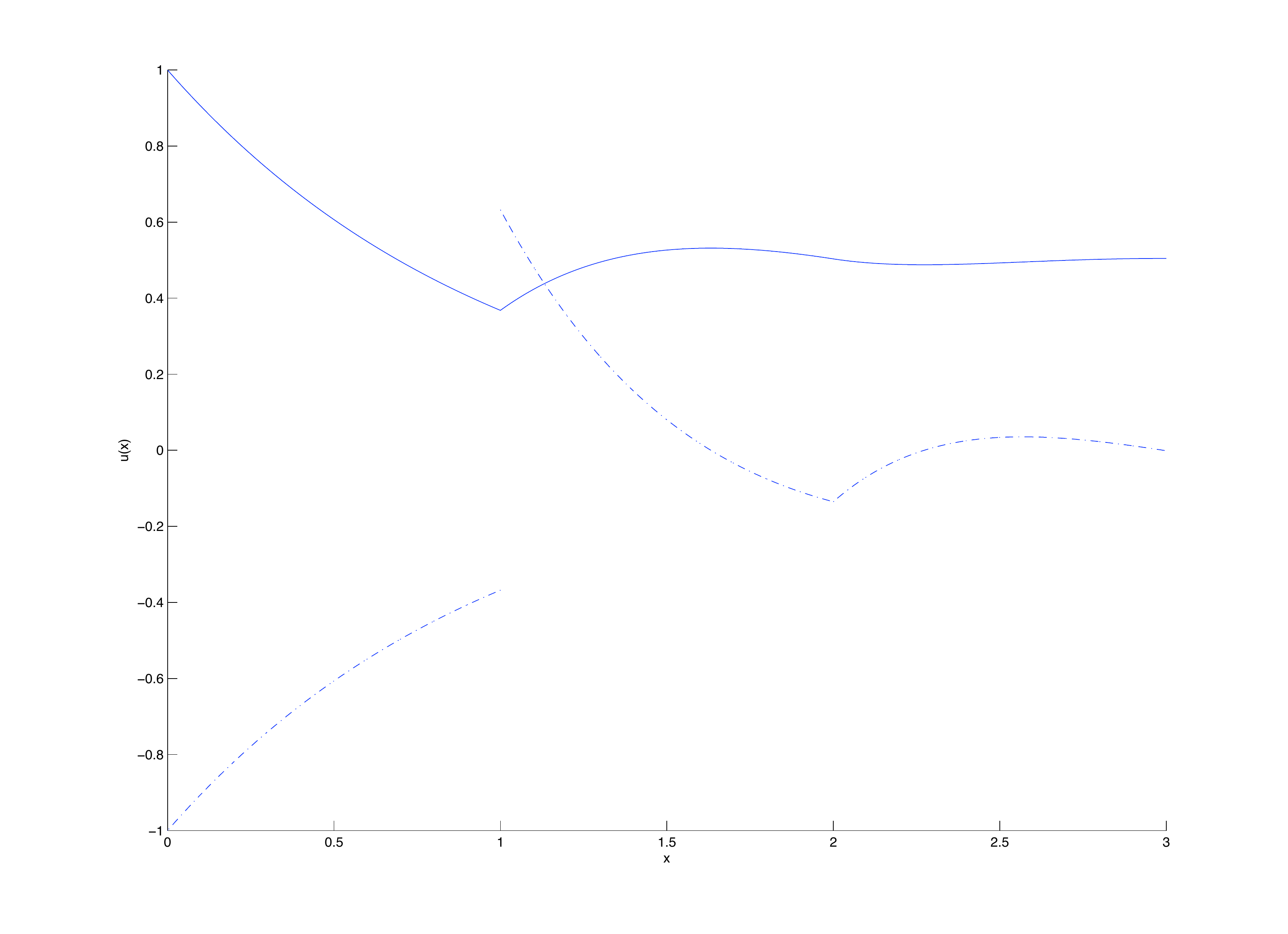}
		        \end{center}
			\caption{$u$ (solid curve), $u'$ (dashed curve) for $\delta=1$ and $\Pi=\delta_1$}
			\label{fig:u}
		\end{figure}
	This first representation already sheds some light on the possible behavior of $u$ for atomic L\'evy measures: in general $u$ is not monotone, is asymptotically linear at zero, and possibly non-differentiable.\\	
	Let us consider the issue of differentiability in more detail. Apparently $u$ is infinitely differentiable in the open intervals $(n,n+1)$, $n=0,1,...$, with
	\begin{align*}
		u'(x)=\begin{cases}
				-e^{-x}&:x<1,\\
				 -e^{-x}+\sum_{i=1}^n\frac{(x-i)^{i-1}}{(i-1)!}e^{-(x-i)}-\sum_{i=1}^n\frac{(x-i)^i}{i!}e^{-(x-i)}&:x\in(n,n+1).
			\end{cases}
	\end{align*}
	The only problematic points are the integers at which the piecewise defined function $u'$ gets the additional summands
	\begin{align*}
		\begin{cases}
			e^{-(x-1)}-(x-1)e^{-(x-1)}&:n=1,\\
			\frac{(x-n)^{n-1}}{(n-1)!}e^{-(x-n)}-\frac{(x-n)^n}{n!}e^{-(x-n)}&:n>1,
		\end{cases}
	\end{align*}
	when going from $x<n$ to $x>n$. This sheds some light on differentiability as for $x=1$ the right derivative is $1$ and the left derivative $0$ whereas for $n>1$ the remainder summand vanishes. In particular, this implies that $u$ is differentiable everywhere except at $1$ with
	\begin{align*}
		u'(x+)-u'(x-)=\Pi(\{x\}).
	\end{align*}
	Taking higher order derivatives, the same reasoning shows that at $\{n\}$, $u$ is $(n-1)$-times but not $n$-times differentiable.
\end{example}

	The example reveals three properties which we aim to understand for general subordinators: higher order differentiability of $u$ depends crucially on the atoms of $\Pi$, $u$ is generally not monotone and $u'$ vanishes at infinity and is asymptotically linear at zero.\medskip

	\section{Results}\label{S2}

	Our Fourier analytic approach forces us to control the small jumps. To this end we recall the Blumenthal-Getoor index $\beta(\Pi)$ for subordinators:
	\begin{align*}
		\beta(\Pi)=\inf\Big\{\gamma>0:\int_0^1x^{\gamma}\Pi(dx)<\infty\Big\}.
	\end{align*}
	Some of our main results are formulated under the following assumption.
	\begin{assumption}[\textbf{A}]\label{ass}
		Assume that $\beta(\Pi)<1$.
	\end{assumption}
	Note that property (\ref{c}) implies that $\beta(\Pi)\leq 1$ so that the assumption only rules out boundary cases and in particular any stable subordinator with drift is included. More direct computations for particular examples of interest show that our results will be still valid for Blumenthal-Getoor index $1$.

\subsection{Series and Integral Representations}\label{sec:representations}
	The main objective are the potential densities. Our results are motivated by the following extension of the renewal type equation for $u$.

\begin{lemma}\label{l1}
The $q$-potential measure $U^{(q)}(dx)$ has a density $u^{(q)}$ satisfying the renewal type equation
\begin{align}\label{a2}
	\delta u^{(q)}(x)=1-\int_0^xu^{(q)}(x-y)(\bar\Pi(y)+q)dy.
\end{align}
\end{lemma}

	Iterating the renewal type equation (\ref{a2}), one heuristically obtains a series representation for the potential density $u^{(q)}$. Making this a rigorous statement is slightly involved as $\bar\Pi$ might have a singularity at zero. The problems caused by the singularity can be circumvented taking into account only the local integrability of $\bar\Pi$ at zero which is a direct consequence of the property (\ref{c}). In \cite{CKS10} a proof for the following proposition was carried out for non-atomic L\'evy measures. In fact this proposition can be recovered from Chapter 3 in \cite{K72} where it is disguised in terms of $p$-functions. For completeness we sketch a proof below.
		
	\begin{proposition}[Series Representation for $u^{(q)}$]\label{prop1}
		The series representation
		\begin{align}\label{series}
			u^{(q)}(x)=\sum_{n= 0}^{\infty}\frac{(-1)^n}{\delta^{n+1}}\big(\1\ast \big(\bar\Pi+q\big)^{\ast n}\big)(x)
		\end{align}
		holds for the potential densities $u^{(q)}$. Here, $\1$ denotes the Heavyside function.
	\end{proposition}

	At this point we should note that for $q=0$, Equation (\ref{series}) appears in a slightly different form in \cite{CKS10}. This is due to a different definition of convolutions. In this paper we define the convolution of two real functions $f$ and $g$ by
	\begin{align*}
		f\ast g(x)=\int_0^xf(x-y)g(y)\,dy,
	\end{align*}
	(and $\mathbf 1\ast f^{\ast 0}(x)=1$) whereas in \cite{CKS10} the convolution was defined by $f\ast g(x)=\int_0^xf(x-y)g'(y)\,dy$. In any case, with $F(x)=\int_{0}^{x}f(t)dt$ and $G(x)=\int_{0}^{x}g(t)dt$, we have $F\ast G=\1\ast f\ast g$, where the second convolution is in the sense of the present paper.\medskip
	
	Unfortunately, deriving properties of $u^{(q)}$ from (\ref{series}) is a delicate matter as one has to deal with an infinite alternating sum of iterated convolutions. Nevertheless, one particular property that follows from (\ref{series}) is that $u^{(q)}$ is of bounded variation.
				
	\begin{corollary}\label{t1}
		There are increasing functions $u_1^{(q)}$ and $u_2^{(q)}$ such that
		\begin{align*}
			u^{(q)}=u_1^{(q)}-u_2^{(q)}.
		\end{align*}
		In particular, $u^{(q)}$ is a function of bounded variation.
	\end{corollary}
	
	To obtain a deeper understanding, we introduce a more carefull representation for $u^{(q)}$ based on Fourier analysis. In the following we denote by
	\begin{align*}
	 	\mathcal L f(s)=\int_{{-\infty}}^{\infty}e^{-sx}f(x)dx
	\end{align*}
	the bilateral Laplace transform for complex numbers $s=\lambda+i\theta$. As $\bar\Pi(x)$ is only defined for $x>0$, we put $\bar\Pi(x)=0$ for $x\leq 0$.			
	\begin{proposition}[Laplace Inversion Representation of $u^{(q)}$]\label{laplaceinversionu}
		Under Assumption \textbf{(A)}, for arbitrary integer $N$ larger than $0$ and any fixed $\lambda>0$,
		\begin{align*}
			g^{N,q}(\lambda+i\theta)=\frac{\big(-\mathcal L (\bar\Pi+q)(\lambda+i\theta)\big)^N}{(\lambda+i\theta)\delta^{N+1}\big(1+\frac{1}{\delta}\mathcal L (\bar\Pi+q)(\lambda+i\theta)\big)}
		\end{align*}
		is a well defined absolutely integrable function in $\theta$ and
		\begin{align}\label{invers}
			u^{(q)}(x)&=\sum_{n=0}^{N-1}\frac{(-1)^n}{\delta^{n+1}}\1\ast\big(\bar\Pi+q\big)^{\ast n}(x)+e^{\lambda x}\frac{1}{2\pi }\int_{-\infty}^{\infty}e^{i\theta x}g^{N,q}(\lambda+i\theta)\,d\theta.
		\end{align}
	\end{proposition}
	There is a good reason to not only consider the simplest case $N=1$. Larger $N$ makes the integrand decay faster which is needed to derive higher order differentiability via interchanging differentiation and integration. For $N>1$ it is still easy to work with representation (\ref{invers}) as the finite sum can be tackled termwise and for the integral classical convergence results can be applied.\medskip
	
	As a first application of Proposition \ref{laplaceinversionu} we derive a representation for the first derivative of the potential densities. In contrast to Proposition \ref{laplaceinversionu} we are not free to choose $N=1$; the minimal $N$ now depends on the Blumenthal-Getoor index.
			
	\begin{corollary}\label{t2}
		The right and left derivatives of $u^{(q)}$ exist and are given by
			\begin{align}
			\big({u^{(q)}}\big)'(x+)&=-\frac{\bar\Pi(x+)+q}{\delta^2}+\sum_{n=2}^{\infty}\frac{(-1)^n}{\delta^{n+1}} \big(\bar\Pi+q\big)^{\ast n}(x),\label{a}\\
			 \big({u^{(q)}}\big)'(x-)&=-\frac{\bar\Pi(x-)+q}{\delta^2}+\sum_{n=2}^{\infty}\frac{(-1)^n}{\delta^{n+1}}\big(\bar\Pi+q\big)^{\ast n}(x).\label{b}
		\end{align}
		Moreover, under Assumption \textbf{(A)} the following representations hold for $N>-\frac{1}{\beta(\Pi)-1}$ and $\lambda>0$:
		\begin{align}
			 \big({u^{(q)}}\big)'(x+)&=-\frac{\bar\Pi(x+)+q}{\delta^2}+\sum_{n=2}^{N-1}\frac{(-1)^n}{\delta^{n+1}}\big(\bar\Pi+q\big)^{\ast n}(x)+e^{\lambda x}\frac{1}{2\pi}\int_{-\infty}^{\infty}e^{i\theta x}h^{N,q}(\lambda+i\theta)\,d\theta,\label{0}\\
			 \big({u^{(q)}}\big)'(x-)&=-\frac{\bar\Pi(x-)+q}{\delta^2}+\sum_{n=2}^{N-1}\frac{(-1)^n}{\delta^{n+1}}\big(\bar\Pi+q\big)^{\ast n}(x)+e^{\lambda x}\frac{1}{2\pi}\int_{-\infty}^{\infty}e^{i\theta x}h^{N,q}(\lambda+i\theta)\,d\theta.\label{00}
		\end{align}	
		For every $\lambda>0$ the function
		\begin{align*}
			h^{N,q}(\lambda+i\theta)=\frac{\big(-\mathcal L (\bar\Pi+q)(\lambda+i\theta)\big)^N}{\delta^{N+1}\big(1+\frac{1}{\delta}\mathcal L (\bar\Pi+q)(\lambda+i\theta)\big)}
		\end{align*}
		is absolutely integrable in $\theta$.
	\end{corollary}
	\begin{rem}\label{Rem1}
		We should note that a refinement of the approach of \cite{CKS10} for non-atomic L\'evy measures yields the series representation (\ref{a}), (\ref{b}) without assuming Blumenthal-Getoor index smaller than $1$. Indeed, the difference is that in our setting $\bar\Pi$ can have jumps and therefore is outside of the scope of \cite{CKS10}. However, the jumps do not affect the study of the uniform and absolute convergence of the series in \eqref{a} and \eqref{b}, and it can be carried out as in \cite{CKS10}.\\
		The arguments of \cite{CKS10} are not repeated here; we only prove the more usefull Laplace inversion representation under Assumption \textbf{(A)} and derive from this the series representation.
	\end{rem}

	\begin{rem}\label{Rem2}
		We do not know how to derive the asymptotics of $u'$ at infinity only from the series representation (\ref{a}), (\ref{b}). It seems that a representation of the type \eqref{0}, \eqref{00} is much more useful for studying the properties of $u'$ and therefore it is an interesting problem to get such a representation  without Assumption \textbf{(A)}.
	\end{rem}
	
	
\subsection{Higher Order (Non)Differentiability for the Potential Densities}\label{sec:differ}

	Having established that $u^{(q)}$ is of bounded variation in Corollary \ref{t1}, $u^{(q)}$ must be differentiable away from a Lebesgue null set. The null set can be identified as the set of atoms as a consequence of either of the 			 representations given in Corollary \ref{t2}.
 	\begin{corollary}\label{uu}
 		The potential densities $u^{(q)}$ are differentiable at $x$ if and only if $x$ is not an atom of $\Pi$. More precisely,
		\begin{align}
			\big({u^{(q)}}\big)'(x+)-\big({u^{(q)}}\big)'(x-)=\frac{\Pi(\{x\})}{\delta^2}.\label{cc}
		\end{align}
	\end{corollary}
	It is interesting to see that the derivative of $u^{(q)}$ only jumps upwards and how the size of the jumps is determined by the weight of the atoms and the drift. The reader might want to compare this with Figure \ref{fig:u}.\medskip
	
	\begin{center}\textbf{Unless explicitly mentioned, from now on we assume that Assumption \textbf{(A)} holds.}\end{center}\medskip
	
	The remaining part of this section deals with higher order differentiability properties of the $q$-potentials. The study focuses on subordinators whose L\'evy measure has atoms which accumulate at most at zero. To emphasise the effect of atomic L\'evy measure on differentiability, the  results are divided into three theorems for non-atomic L\'evy measure, purely atomic L\'evy measure, and L\'evy measure with atomic and absolutely continuous parts.\medskip
	
	The case of non-atomic L\'evy measure already appeared in \cite{CKS10}. For Blumenthal-Getoor index smaller than $1$ their results follow easily from the Laplace inversion representation. To present a complete picture we 	state and reprove the following result of \cite{CKS10} for smooth L\'evy measures.

	\begin{thm}\label{t12}
		If $\bar\Pi$ is everywhere infinitely differentiable, then $U^{(q)}$ is everywhere infinitely differentiable.
	\end{thm}
	
	More accurate effects of the atoms on differentiability are revealed in what follows. The explicit calculation of Example \ref{example} for $\Pi=\delta_1$ indeed suggests more interesting behavior for higher order derivatives. In that particular case, $u$ is $(n-1)$-times differentiable but not $n$-times differentiable at $n$. The critical points $n\in\N$ in the example have the property that they can be reached by precisely $n$ jumps of the size of the atom $1$.\\
	Denote by $G$ the atoms of the L\'evy measure $\Pi$. We say that $x>0$ can be reached by $n$ atomic jumps if $x=\sum_{i=1}^ng_i$ with $g_i\in G$. For $k\geq 1$ we define the sets
	\begin{align*}
		G_k=\Bigg\{x> 0:x=\sum_{i=1}^jg_i, g_i\in G, 1\leq j\leq k\Bigg\}
	\end{align*}
	of points that can be reached by at most $k$ atomic jumps. The next theorem shows how higher order differentiability is connected to the sets $G_k$. As it is formulated for purely atomic L\'evy measures it can be seen as the counterpart of Theorem \ref{t12}.	
	
	\begin{thm}\label{t4}
		If $\Pi$ is purely atomic with a possible accumulation point of atoms only at zero, then, for $k\geq 1$,
		\begin{align*}
			U^{(q)}\text{ is }(k+1)\text{-times differentiable at }x\qquad \Longleftrightarrow \qquad x\notin G_k.
		\end{align*}
		In particular, if $x$ cannot be reached by atomic jumps only, then $U^{(q)}$ is infinitely differentiable at $x$.
	\end{thm}
	
	The following example shows that we can easily find examples with exotic differentiability properties.
	\begin{example}
		If $\Pi$ has atoms on $1/k$ for $k\geq 1$, then $U^{(q)}$ is infinitely differentiable at $x$ if and only if $x\notin \mathbb{Q}$. For $x=n/k$, $U^{(q)}$ is at most $n$-times differentiable at $x$.
	\end{example}
	Finally, we show how Theorems \ref{t12} and \ref{t4} combine each other: the atoms prevent $U^{(q)}$ from being twice differentiable and an additional absolutely continuous part does not change the behavior.
	\begin{thm}\label{t9}
		If $\bar \Pi(x)=\bar \Pi_1(x)+\bar\Pi_2(x)$, where $\Pi_1$ is purely atomic with possible accumulation point of atoms only at zero and $\bar\Pi_2$ is infinitely differentiable, then, for $k\geq 1$,
		\begin{align*}
			U^{(q)}\text{ is }(k+1)\text{-times differentiable at }x\qquad \Longleftrightarrow \qquad x\notin G_k.
		\end{align*}
		In particular, if $x$ cannot be reached by atomic jumps only, then $U^{(q)}$ is infinitely differentiable at $x$.
	\end{thm}
	\begin{rem}\label{Rem3}
	In fact Theorem \ref{t9} is still true if $\bar\Pi_2$ possesses only finitely many derivatives. Then the statement will be valid, for any $k$, such that $\bar\Pi_2$ is $k$-times differentiable.
	\end{rem}
	\begin{rem}
		Recalling the close connection of $u$ and creeping probabilities given in (\ref{creep}), it would be interesting to have a good probabilistic interpretation of non-existence of derivatives of certain order for the creeping probabilities.
	\end{rem}


\subsection{Asymptotic Properties of Potential Densities}\label{sub2}
	Asymptotic properties at zero and infinity of $U$ are classical results in potential theory of subordinators (see for instance Proposition 1 of Chapter 3 in \cite{B96}). Also the limiting behavior of the renewal density $u$ at zero and infinity are known (see \eqref{d} and  \eqref{dd}).
	
	In what follows we aim at giving more refined convergence properties based on the representations for $u^{(q)}$ and $(u^{(q)})'$. Some results will be only valid for $q=0$. Utilizing the series representation of Section \ref{sec:representations}, we first strengthen the asymptotic at zero. In the following $f \sim g$ denotes strong asymptotic equivalence, i.e. $\lim \frac{f}{g}=1$.
	\begin{thm}\label{p2}
		For general subordinator with positive drift and $n\geq 0$, the following strong asymptotic equivalence holds:
		\begin{align*}
			\Bigg|\sum_{k=0}^{{n}}\frac{(-1)^k}{\delta^{k+1}}\1\ast (q+\bar\Pi)^{\ast k}(x)-u^{(q)}(x)\Bigg|\sim \frac{(-1)^{n}}{\delta^{n+2}}\1\ast (q+\bar\Pi)^{\ast (n+1)}(x), \quad\text{as $x\rightarrow0$}.
		\end{align*}
	\end{thm}
	The first order asymptotic of (\ref{d}) appears in the theorem for $n=0$ and reveals more precise qualitative information:	
	
	\begin{corollary}\label{c1}
	The potential density $u^{(q)}$ is asymptotically linear with slope $-(\Pi(\R)+q)/\delta^2$ at zero iff the L\'evy measure is finite. If the L\'evy measure is infinite, then $u^{(q)}$ decays faster than linearly.
	\end{corollary}
	A simple class of examples are stable subordinators with drift.
	\begin{example}
		If the L\'evy measure has stable law with $\alpha\in (0,1)$, the strong asymptotics at zero are given by
		\begin{align*}
			\frac{1}{\delta}-u(x) \sim \frac{1}{\delta^2}C\int_0^x\frac{1}{y^{\alpha}}\,dy=\frac{C}{\delta^2(1-\alpha)}x^{1-\alpha}.
		\end{align*}
	\end{example}	
	L\'evy measure putting mass precisely to one point provides an example of asymptotically linear behavior.
	\begin{example}
		If the L\'evy measure has only one atom at $1$ and $\delta=1$, then the direct calculation in Example \ref{example} or Corollary \ref{c1} both lead to the linear asymptotics
		\begin{align*}
			1-u(x)=1-e^{- x}\sim x,\quad\text{as }x\rightarrow 0.
		\end{align*}	
	\end{example}
	As for absolutely continuous L\'evy measures the existence of the derivative of the renewal density was established quite recently, the asymptotics of ${u}'$ seem to be unknown. On first view one might be tempted to differentiate the asymptotics of $u$ to derive the asymptotics of ${u}'$. Due to lack of knowledge on ultimate monotonicity of $u$ and ${u}'$ we cannot apply the monotone density theorem and instead work with the representations of Section \ref{sec:representations}.
	\begin{thm}\label{t7}
		Assume condition \textbf{(A)}, i.e. $\beta(\Pi)<1$. Let $\bar\Pi(0+)=\infty$, then there exists $n\geq 1$ such that $\beta(\Pi)\leq \frac{n}{n+1}$ and then
		\begin{align*}
			\big({u^{(q)}}\big)'(x+)&=-\frac{\bar\Pi(x+)}{\delta^2}+...+ \frac{(-1)^n\bar\Pi^{\ast n}(x)}{\delta^{n+1}}+o(1),\quad\text{ as }x\to 0,\\
			\big({u^{(q)}}\big)'(x-)&=-\frac{\bar\Pi(x-)}{\delta^2}+...+ \frac{(-1)^n\bar\Pi^{\ast n}(x)}{\delta^{n+1}}+o(1),\quad\text{ as }x\to 0.
		\end{align*}
If $\bar\Pi(0+)<\infty$ then, as $x\to 0$,
\[\big({u^{(q)}}\big)'(x+)=\frac{q+\bar\Pi(x+)}{\delta^2}+o(1),\]
\[{\big(u^{(q)}}\big)'(x-)=\frac{q+\bar\Pi(x-)}{\delta^2}+o(1).\]
	\end{thm}

	For a large class of processes this implies that the asymptotic of $({u}^{(q)})'$ are indeed given by the derivative of the asymptotic of $u^{(q)}$:
	\begin{corollary}
		If $\Pi$ has Blumenthal-Getoor index $\beta<1/2$, then $\big({u^{(q)}}\big)'(x+)\sim -\frac{q+\bar\Pi(x+)}{\delta^2}$ and $\big({u^{(q)}}\big)'(x-)\sim -\frac{(q+\bar\Pi)(x-)}{\delta^2}$ at zero.
	\end{corollary}
	
	The hypothesis of the theorem is not necessary to have the asymptotic $\bar\Pi(x)/\delta^2$ at zero as the next example shows.
	\begin{example}	
		If $X$ is stable with index $\alpha\in (0,1)$ and drift $\delta>0$, then
		\begin{align*}
			\bar\Pi^{\ast 2}(x)=C_{1}\int_0^x (x-y)^{-\alpha}y^{-\alpha}\,dy=C_{2} x^{-2\alpha+1}
		\end{align*}
		and inductively $\bar\Pi^{\ast n}(x)=C_{n} x^{-n\alpha+n-1}$. This, combined with Theorem \ref{t7}, implies that $${u}'(x)\sim -\frac{\bar\Pi(x)}{\delta^2}=-\frac{C_{1}x^{-\alpha}}{\delta^2},\quad\text{ as }x\to 0.$$
	\end{example}
	Finally, we consider the asymptotics of ${u}'$ at infinity which is technically more involved as we could not derive the asymptotics from the series representation (\ref{a}), (\ref{b}). Instead, the proof is based on a refinement of the Laplace inversion representation.
	
	\begin{thm}\label{t8}
		If $\E[X_1]<\infty$, then
		\begin{align*}
			\lim_{x\to \infty}{u}'(x+)=\lim_{x\to \infty}{u}'(x-)=0.
		\end{align*}
	\end{thm}
	As our approach does not work for infinite mean subordinators, it remains open whether ${u}'$ vanishes at infinity or not in this setting.
\section{Proofs}\label{S3}
				
	\subsection{Series and Integral Representations}
	\begin{proof}[Proof of Lemma \ref{l1}]
		When $q=0$ the statement was already mentioned in Section \ref{MM1} and is due to Kesten (see \cite{K69}). Let us now fix $q>0$. First note that for any random variable $e_{q}\sim Exp(q)$  which is independent of $X$ using the Markov property at time $e_{q}$ and integrating out the possible positions of $X_{e_{q}}\leq x$ we get
\[U(x)=\E\Big(\int_{0}^{e_{q}}\mathbf{1}_{\{X_{t}\leq x\}}dt\Big)+\int_{0}^{x}U(x-y)\P(X_{e_{q}}\in\,dy)=U^{(q)}(x)+qU*u^{(q)}(x),\]
since $u^{(q)}$ is continuous and bounded whenever the subordinator possesses a positive drift, see (iii) in Theorem 16 in \cite{B96}. Therefore, differentiation yields
\[u(x)=u^{(q)}(x)+qu*u^{(q)}(x).\]
Comparing the latest with
\begin{align*}
	\delta u(x)&=\P(X_{T_{x}}=x)\\
	&=\P(X_{T_{x}}=x;\,T_{x}\leq e_{q})+\int_{0}^{x}\P(X_{T(x-y)}=x-y)\P(X_{e_{q}}\in dy)\\
	&=\P(X_{T_{x}}=x;\,T_{x}\leq e_{q})+q\delta \int_{0}^{x}u(x-y)u^{(q)}(y)dy,
\end{align*}
 we get the expected relation
$$\P(X_{T_{x}}=x;\,T_{x}\leq e_{q})=\delta u^{(q)}(x).$$
We use this final relation to write
\begin{align*}
\P(T_{x}\leq e_{q})=1-\P(X_{e_{q}}\leq x)=\delta u^{(q)}(x)+\P(X_{T_{x}}>x;\,T_{x}\leq e_{q}).
\end{align*}
Next we have, using the compensation formula from Chapter 0 in \cite{B96},
	\begin{align*}
		\P(X_{T_{x}}>x;\,T_{x}\leq e_{q})&=q\int_{0}^{\infty}e^{-qt}\E\Big(\sum_{s\leq t}\mathbf{1}_{\{X_{s-}\leq x;\,X_{s}>x\}}\Big)dt\\
		&=q\int_{0}^{\infty}e^{-qt}\int_{0}^{t}\int_{0}^{x}\bar \Pi(x-y)\P(X_{s}\in dy)ds\\
		&=\int_{0}^{x}\bar\Pi(x-y)u^{(q)}(y)dy.
\end{align*}
This, combined with $\P(X_{e_{q}}\leq x)=q\int_{0}^{x} u^{(q)}(y)dy$, proves the assertion.
	\end{proof}
	\begin{proof}[Proof of Proposition \ref{prop1}:]	
	 	Since $\bar\Pi$ may not be continuous, we cannot directly apply the general Theorem 6 of \cite{CKS10}. Nonetheless, the proof of Theorem 6 of \cite{CKS10} can be adapted to this situation by noting that in fact it is not crucial that $g$ there is continuously differentiable but the fact that it is almost everywhere differentiable with a negative derivative vanishing at infinity. The latter is due to the fact that $g'$ is used in convolutions. In this case, for completeness, we sketch a proof but we refer to \cite{CKS10} for details. First note that each summand of
	\begin{align}
		\phi(x):=\sum_{n=0}^{\infty}(-1)^n\frac{\1\ast\big(\bar \Pi+q\big)^{\ast n}(x)}{\delta^{n+1}}\label{equ:b'},
	\end{align}
	is finite as for subordinators $\int_0^{1}x\Pi(dx)$ is finite showing that $\bar\Pi$ is not exploding too fast at zero. We now show that the series is absolutely converging in particular showing that it is well defined and can be rearranged. First, since $\1\ast\big(\bar\Pi+q\big)^{\ast(n-1)}(x)$ is increasing we get iteratively
	\begin{align}\label{iter}
		{\1\ast \big(\bar\Pi+q\big)^{\ast n}(x)}\leq \big(\1\ast \big(\bar\Pi+q\big)(x)\big)^n.
	\end{align}
	As $\1\ast(\bar\Pi+q)(x)$ is continuously increasing and $\1\ast(\bar\Pi+q)(0)=0$ there is $b>0$ such that for all $x\leq b$,  $\1\ast(\bar\Pi+q)(x)<\delta/2$ and hence the series defining $\phi$ is absolutely and uniformly convergent. Now we use an iterative procedure to extend the absolute convergence to all $x>0$: reasoning as above we obtain
	\begin{align*}
		\frac{\1\ast\big(\bar\Pi+q\big)^{\ast n}(2b)}{\delta^{n+1}}&=\int_0^b\frac{\1\ast \big(\bar\Pi+q\big)^{\ast(n-1)}(2b-y)}{\delta^{n}}\frac{(\bar\Pi+q)(y)}{\delta}\,dy\\
		&\quad+\int_b^{2b}\frac{\1\ast \big(\bar\Pi+q\big)^{\ast(n-1)}(2b-y)}{\delta^n}\frac{(\bar\Pi+q)(y)}{\delta}\,dy\\
		&\leq \frac{\1\ast \big(\bar\Pi+q\big)^{\ast(n-1)}(2b)}{\delta^n}\frac{\1\ast (\bar\Pi+q)(b)}{\delta}+ \frac{\1\ast \big(\bar\Pi+q\big)^{\ast(n-1)}(b)}{\delta^n}\frac{\1\ast(\bar\Pi+q)(2b)}{\delta}\\
		&\leq(n+1)\left(\frac{\1\ast (\bar\Pi+q)(b)}{\delta}\right)^{(n-1)}\frac{\1\ast (\bar\Pi+q)(2b)}{\delta}\leq \frac{(n+1)}{2^{n-1}}\frac{\1\ast (\bar\Pi+q)(2b)}{\delta}.
	\end{align*}
	Iterating this arguments as in the proof of Theorem 6 of \cite{CKS10} shows locally uniform and absolute convergence of the series for all $x>0$.
	To verify that $\phi$ equals $u$ we exploit Fubini's theorem to obtain
	\begin{align*}
		\phi(x)=\frac{1}{\delta}-(\bar\Pi+q)\ast\frac{1}{\delta}\sum_{n=1}^{\infty}(-1)^{n-1}\frac{\1\ast\big(\bar \Pi+q\big)^{\ast(n-1)}(x)}{\delta^{n}}=\frac{1}{\delta}-\frac{1}{\delta}(\bar\Pi+q)\ast \phi(x),
	\end{align*}
	implying that $\phi$ is a solution of Equation (\ref{a2}). Uniqueness of solutions for locally bounded functions follows since for two solutions $f$ and $\phi$
	\begin{align*}
		|f-\phi|(x)&\leq|(f-\phi)\ast 1/\delta(\bar\Pi+q)(x)|\\
		&=...\\
		&=|(f-\phi)\ast 1/\delta^k\big(\bar\Pi+q\big)^{\ast k}(x)|\\
		&\leq \sup_{s\leq x}|f-\phi|(s)1/\delta^k\1\ast(\bar\Pi+q)^{\ast k}(x)
	\end{align*}
	for all $k$. The right hand side converges to zero since $\delta^{-k}\1\ast\big(\bar\Pi+q\big)^{\ast k}(x)$ goes to zero for all $x\geq 0$ as shown above since it is a member of uniformly convergent series.
	\end{proof}

\begin{proof}[Proof of Corollary \ref{t1}]
	The proof follows directly from Proposition \ref{prop1}. To define the functions $u_1^{(q)}$ and $u_2^{(q)}$ we separate the positive and negative parts as
	\begin{align*}
		u^{(q)}(x)=\sum_{n=0}^{\infty}\frac{\1\ast\big(\bar \Pi+q\big)^{\ast 2n}(x)}{\delta^{n+1}}-\sum_{n=0}^{\infty}\frac{\1\ast\big(\bar \Pi+q\big)^{\ast (2n+1)}(x)}{\delta^{(2n+1)}}=:u^{(q)}_1(x)-u^{(q)}_2(x)
	\end{align*}
	which is possible taking into account the absolute convergence of the series representation. Each summand of the defining series for $u^{(q)}_1$ and $u^{(q)}_2$ is increasing, thus, $u^{(q)}_1$ and $u^{(q)}_2$ themselves are increasing.
\end{proof}

		We are now going to derive the Laplace transform representation of Proposition \ref{laplaceinversionu}. As mentioned above, the main difficulty is that we need to deal with convolutions of possibly unbounded functions.
		Let us first motivate the approach. Taking into account the series representation we divide $u^{(q)}$ into two parts:
		\begin{align*}
 			u^{(q)}(x)&=\sum_{n=0}^{N-1} \frac{(-1)^n}{\delta^{n+1}}\1\ast\big(\bar\Pi+q\big)^{\ast n}(x)+\sum_{n\geq N}\frac{(-1)^n}{\delta^{n+1}}\1\ast\big(\bar\Pi+q\big)^{\ast n}(x)\\
 			&=:\sum_{n=0}^{N-1} \frac{(-1)^n}{\delta^{n+1}}\1\ast\big(\bar\Pi+q\big)^{\ast n}(x)+\phi^{N,q}(x).
		\end{align*}
		The goal of the following Fourier analysis is to derive a convenient integral representation for $\phi^{N,q}$. For this sake first note that it follows as in the proof of Proposition \ref{prop1} that $\phi^{N,q}$ is the unique locally 			 bounded solution of the following renewal type equation:
		\begin{equation}\label{LEq1}
			 \phi^{N,q}(x)=\frac{(-1)^{N}}{\delta^{N+1}}\1\ast\big(\bar\Pi+q\big)^{*N}(x)-\frac{1}{\delta}\phi^{N,q}(x)*(\bar\Pi+q)(x).
		\end{equation}
		If we were allowed to turn convolution into multiplication in Laplace domain for $Re(s)>0$, Equation (\ref{LEq1}) transforms into
		\begin{align}\label{conve}
			\mathcal L \phi^{N,q}(s)=\frac{(-1)^{N}}{\delta^{N+1}} \mathcal L \1(s)\mathcal L(\bar\Pi+q)(s)^N-\frac{1}{\delta}\mathcal L\phi^{N,q}(s) \mathcal L(\bar \Pi+q)(s).
		\end{align}
		Solving (\ref{conve}) for $\mathcal L\mathcal \phi^{N,q}$, leads to
		\begin{align}\label{g}
	 		\mathcal L \phi^{N,q}(s)= \frac{(-1)^N\mathcal L \1(s)(\mathcal L(\bar\Pi+q)(s))^N}{\delta^{N+1}(1+\frac{1}{\delta}\mathcal L(\bar \Pi+q)(s))}=:g^{N,q}(s)
		\end{align}
		whenever the quotient is well-defined. If furthermore we were able to verify integrability conditions needed for Laplace inversion we obtain an integral representation for $\phi^{N,q}$. We are now going to check what is needed to turn this formal approach into rigorous statements.\medskip
		
		\begin{center}\textbf{For the rest of this section we set $\delta=1$ in order to simplify the notation.}\end{center}\medskip
		
		For the first step it is shown that indeed Equation (\ref{LEq1}) turns into Equation (\ref{conve}). A priori this is not clear due to the possible singularity of $\bar\Pi$ at zero.
		
		\begin{lemma}\label{l33}
			There is $\lambda_0\geq 0$ such that $\mathcal L \phi^{N,q}$ solves (\ref{conve}) on $\{\lambda+i\theta:\lambda\geq \lambda_0, \theta\in\R\}$.
		\end{lemma}

		\begin{proof}
			To show that the first transformation can be carried out we show for $s=\lambda+i\theta$, $\lambda$ bounded from below by some $\lambda_0$, that $\mathcal L (\bar\Pi+q)(s)$ and $\mathcal L \phi^{N,q}(s)$ are well defined to deduce $\mathcal L(\bar\Pi+q)^{\ast n}(s)=\mathcal L (\bar\Pi+q)(s)^n$, $\mathcal L (\bar\Pi+q)\ast \phi^{N,q}(s)=\mathcal L (\bar\Pi+q)(s)\mathcal L \phi^{N,q}(s)$, and $\mathcal L \1\ast (\bar\Pi+q)(s)=\mathcal L \mathbf 1 (s)\mathcal L (\bar\Pi+q)(s)$.
			
			\smallskip To validate Laplace transformation of $(\bar\Pi+q)^n$ for $\lambda$ large enough, note that we may choose $\lambda_0$ such that
			\begin{align}\label{11}
				\int_0^{\infty}e^{-\lambda_0 x} (\bar\Pi+q)(x)\,dx<1
			\end{align}
			which is possible as $\int_0^1\bar\Pi(x)\,dx<\infty$ and $\lim_{x\to\infty}\bar\Pi(x)=0$. It now follows directly from Fubini's theorem that iterated convolutions of $\bar\Pi+q$ turn into multiplication under Laplace transforms. We now show that $\phi^{N,q}$ can be Laplace transformed for which we use the fact that
			\begin{align*}
				\sum_{n=N}^{\infty} \big|\mathcal L \big(\bar\Pi+q\big)^{\ast n}(s)\big|=\sum_{n=N}^{\infty} \big|\mathcal L (\bar\Pi+q)(s)\big|^n<\infty
			\end{align*}
			to justify the change of summation and integration in the following:
			\begin{align}\label{12}
				\mathcal L \phi^{N,q}(s)&=\int_0^{\infty}e^{-sx}\sum_{n= N}^{\infty}(-1)^n\big(\bar\Pi+q\big)^{\ast n}(x)\,dx
				=\sum_{n= N}^{\infty}(-1)^n\int_0^{\infty}e^{-sx}\big(\bar\Pi+q\big)^{\ast n}(x)\,ds<\infty.
			\end{align}
			Now it is only left to show $\mathcal L \phi^{N,q}\ast (\bar\Pi+q)(s)=\mathcal L \phi^{N,q}(s)\mathcal L (\bar \Pi+q)(s)$ and further $\mathcal L \1\ast (\bar\Pi+q)(s)=\mathcal L \1 (s)\mathcal L (\bar\Pi+q)(s)$. As $\bar\Pi+q$ and $\phi^{N,q}$ can be Laplace transformed, we obtain from (\ref{11}) and (\ref{12}) the bound
			\begin{align*}
				&\quad\int_0^{\infty}\int_0^{\infty}\Big|e^{-s t}\phi^{N,q}(t-x)(\bar\Pi+q)(x)\Big|\,dtdx\\
				&=\int_0^{\infty}\int_0^{\infty}\Big|e^{-\lambda (t-x)}\phi^{N,q}(t-x)\Big|\Big|e^{-\lambda x}(\bar\Pi+q)(x)\Big|\,dtdx<\infty
			\end{align*}
			enabling us to apply Fubini's theorem once more to obtain $\mathcal L \phi^{N,q}\ast (\bar\Pi+q)(s)=\mathcal L \phi^{N,q}(s)\mathcal L (\bar\Pi+q)(s)$.
			The final identity $\mathcal L \1\ast (\bar\Pi+q)(s)=\mathcal L \1 (s)\mathcal L (\bar\Pi+q)(s)$ follows from similar arguments noting that $|\mathcal L \1(s)|\leq \frac{1}{\lambda}$.
		\end{proof}		
		The second step in our analysis consists of showing that the convolution Equation (\ref{conve}) can indeed be solved in Laplace domain leading to Equation (\ref{g}). The following lemma is stronger then the previous as we can show that $g^{N,q}(s)$ is well defined for any $s=\lambda+i\theta$ with $\lambda>0$ even though a priori we do not know that $g^{N,q}$ is the Laplace transform of $\phi^{N,q}$.
			
	\begin{lemma}\label{l11}
		Suppose that $\Pi$ is non-trivial, then $g^{N,q}(\lambda+i\theta)$ is well-defined for $\lambda>0$.
	\end{lemma}
	
	\begin{proof}
		To show that $g^{N,q}(s)$ is well-defined at $s\in\mathbb{C}$ with positive real-part, it suffices to show that
		\begin{align*}
			\mathcal L(\bar\Pi+q)(s)\neq -1.
		\end{align*}
		Let us assume that $\mathcal L(\bar\Pi+q)(s_0)=-1$ for some $s_{0}=\lambda_0+i\theta_0$ with $\lambda_0>0$. Without loss of generality way may assume $\theta_0\neq0$ as otherwise the contradiction follows trivally.			The assumption necessarily implies that
		\begin{align}\label{cd}
			Im(\mathcal L(\bar\Pi+q)(s_0))=\int_{0}^{\infty}e^{-\lambda_0 x}\sin(\theta_0 x)(\bar\Pi+q)(x)dx=0.
		\end{align}
		As $\bar\Pi$ is decreasing, we see that
		\begin{align}
			\int_{0}^{\infty}e^{-\lambda_0 x}(\bar\Pi+q)(x)\,dx<\infty.\label{neu}
		\end{align}
		Dividing the integral of the absolutely integrable function $e^{-\lambda_0 x}(\bar\Pi+q)(x)\sin(\theta_0 x)$ into pieces of the length of one period of the sine function and applying Fubini's theorem, we obtain
		\begin{align*}
			0&=\int_0^{\infty}e^{-\lambda_0 x}(\bar\Pi+q)(x)\sin(\theta_0 x)dx\\
			 &=\int_{2\pi/\theta_0}^{\infty}\sum_{k=0}^{\infty}\mathbf{1}_{[2k\pi/\theta_0,2(k+1)\pi/\theta_0)}(x)e^{-\lambda_0 x}(\bar\Pi+q)(x)\sin(\theta_0 x)dx\\
			&=\sum_{k=0}^{\infty}\int_{\frac{2k\pi}{\theta_0}}^{\frac{2(k+1)\pi}{\theta_0}}e^{-\lambda_0 x}(\bar\Pi+q)(x)\sin(\theta_0 x)dx.
		\end{align*}
		As $e^{-\lambda_0 x}\bar\Pi(x)$ is strictly decreasing, unless $(\bar\Pi+q)(x)=0$ and $(\bar\Pi+q)(x)$ is non-increasing each summand
		\begin{align*}
			\int_{\frac{2k\pi}{\theta_0}}^{\frac{2(k+1)\pi}{\theta_0}}e^{-\lambda_0 x}(\bar\Pi+q)(x)\sin(\theta_{0} x)\,dx
		\end{align*}
		must be non-negative and hence vanish as the total sum is zero. In particular, this implies that $\bar\Pi(x)+q=0$ for all $x>0$ so that for $q>0$ a direct contradiction occurs. For $q=0$ the contradiction occurs as $\Pi$ was assumed to be non-trivial. Thus  $g^{}(s)$ is well-defined.
	\end{proof}
		
	\begin{rem}\label{remark}
		If furthermore $\E[X_1]<\infty$, then $g^{N,0}$ is well-defined also on the imaginary axis. This follows from the same proof noting that in this case (\ref{neu}) holds as well for $\lambda_0=0$. Indeed as each term $\int_{\frac{2k\pi}{\theta_0}}^{\frac{2(k+1)\pi}{\theta_0}}\bar\Pi(x)\sin(\theta_{0} x)\,dx$ has to vanish we conclude that $\Pi$ has to be concentrated on $\{2k\pi/\theta_{0}\}_{k\geq 1}$. On the other hand, in this case, as
		\begin{align*}
		Re(\mathcal L\bar\Pi(s_0))&=\int_{0}^{\infty}e^{-\lambda_0 x}\cos(\theta_0 x)\bar\Pi(x)dx\\
		&=\sum_{k\geq 0}\bar\Pi\Big(\frac{2k\pi}{\theta_{0}}\Big)\int_{\frac{2k\pi}{\theta_0}}^{\frac{2(k+1)\pi}{\theta_0}}\cos(\theta_0 x)dx=0,
		\end{align*}
		we see that $Re(\mathcal L\bar\Pi(s_0))\neq-1$ and thus $g^{N,q}(s)$ is well-defined.
	\end{rem}

		The third step of our derivation of an integral representation for $u^{(q)}$ is an inversion approach for $g^{N,q}$. We now briefly discuss the connection to Fourier transforms which is crucial for the inversion: for integrable functions $f$ define for $x\in \R$
		\begin{align*}
			\mathcal F f(x)=\int_{-\infty}^{\infty}e^{-ix t}f(t)\,dt.
		\end{align*}
		Apparently, the Fourier transform $\mathcal F$ appears when evaluating the Laplace transform on the imaginary line only. Defining the auxiliary function
		\begin{align*}
			r_{\lambda}(x)=e^{-\lambda x}(\bar\Pi+q)(x)
		\end{align*}
		the simple connection is
		\begin{align*}
			\mathcal F r_{\lambda}(\theta)=\mathcal L (\bar\Pi+q)(\lambda+i\theta).
		\end{align*}
		Taking into account this close connection of Laplace and Fourier transforms, classical Fourier inversion for $\lambda>0$ gives the inversion formula (also known as Bromwich integral)
		\begin{align*}
			\phi^{N,q}(x)&=\frac{1}{2\pi}\int_{-\infty}^{\infty}\mathcal L \phi^{N,q}(\lambda+i\theta)e^{(\lambda+i\theta)x}\,d\theta\\
			&=\frac{1}{2\pi}e^{\lambda x}\int_{-\infty}^{\infty}g^{N,q}(\lambda+i\theta)e^{i\theta x}\,d\theta
		\end{align*}
		if $g^{N,q}(\lambda+i\theta)$ is absolutely integrable with respect to $\theta$. To prove the needed integrability we start with a simple estimate.
		
		\begin{lemma}\label{l4}
			For any $a>0$ and $y\leq a$ the estimate $\bar\Pi(y)\leq C(a,\eps)y^{-(\beta(\Pi)+\epsilon)}$ holds for all $\eps>0$ with $\beta(\Pi)+\eps<1$.
		\end{lemma}		
		
		\begin{proof}
			First note that by the definition of the Blumenthal-Getoor index $\int_0^1y^{\beta(\Pi)+\eps}\,\Pi(dy)<\infty$ for any $\eps>0$. The claim follows from the simple observation that for any $\alpha>0$ there is $\delta>0$ such that for any $\tau<\delta$
			\begin{align*}
   				 \alpha\geq\int_{\tau}^{\delta}y^{\beta(\Pi)+\eps}\Pi(dy)\geq\tau^{\beta(\Pi)+\eps}(\bar{\Pi}(\tau)-\bar{\Pi}(\delta))
   			\end{align*}
   			as $\bar\Pi$ is decreasing. Letting $\tau$ go to zero, we deduce $\limsup_{\tau\to 0} \tau^{\beta(\Pi)+\eps}\bar\Pi(\tau)\leq \alpha.$
		\end{proof}
		
		The need for Assumption \textbf{(A)} comes from the following lemma and its consequences.
		
		\begin{lemma}\label{l43}
			For any $\lambda>0$ and $\epsilon>0$ the following estimate holds:
			\begin{align*}
				|\mathcal L \bar\Pi(\lambda+i\theta)|\leq \begin{cases}
				                                           C\frac{1}{\lambda}&:|\theta|\leq1,\\
				                                           C|\theta|^{\beta(\Pi)+\epsilon-1}&:|\theta|>1,
				                                          \end{cases}
 			\end{align*}
 			where $C=C(\eps)>0$.
		\end{lemma}
		
	\begin{proof}
		We estimate the imaginary and real part of $\mathcal L \bar\Pi$ separately. For the imaginary part we first estimate for $\theta> 0$:
		\begin{align*}\label{LEq5}
			\nonumber &\big|Im(\mathcal L\bar\Pi(\lambda+i\theta))\big|\\
			&\nonumber=\theta^{-1}\left|\int_{0}^{\infty}\sin(y)r_{\lambda}(y/\theta)dy\right|\\
			 &\nonumber=\theta^{-1}\left|\sum_{k=0}^{\infty}\int_{2k\pi}^{(2k+1)\pi}(r_{\lambda}(y/\theta^{})-r_{\lambda}((y+\pi)/\theta))\sin(y)dy\right|\\
			 \nonumber&\leq\theta^{-1}\left|\pi\sum_{k=1}^{\infty}r_{\lambda}(2k/\theta)-r_{\lambda}((2k+2)/\theta^{})+\theta^{-1}\int_{0}^{\pi}(r_{\lambda}(y/\theta^{})-r_{\lambda}((y+\pi)/\theta^{}))dy\right|\\\
			&\nonumber\leq\theta^{-1}\int_{0}^{\pi}r_{\lambda}(y/\theta^{})dy\\
			&\leq C\theta^{\beta(\Pi)+\eps-1}\int_{0}^{\pi}y^{-(\beta(\Pi)+\eps)}dy=C\theta^{\beta(\Pi)+\eps-1},
		\end{align*}
		where we have used Lemma \ref{l4} and that $r_{\lambda}$ is decreasing in the last inequality. Unfortunately, this uniform in $\lambda$ upper bound is not suitable for all $\theta$ as the constant of Lemma \ref{l4} explodes as $\theta$ approaches zero. To circumvent this problem we derive a different upper bound that works everywhere equally well but is not uniform in $\lambda$:
		\begin{align*}
			&\big|Im(\mathcal L\bar\Pi(\lambda+i\theta))\big|\\
			&\leq \theta^{-1}\int_{0}^{\infty}r_{\lambda}(y/\theta)dy\\
			 &=\theta^{-1}\int_{0}^{\theta}r_{\lambda}(y/\theta)dy+\theta^{-1}\int_{\theta}^{\infty}r_{\lambda}(y/\theta)dy\\
			&\leq \theta^{-1}\int_{0}^{\theta}\bar\Pi(y/\theta)dy+\theta^{-1}\bar\Pi(1)\int_{\theta}^{\infty}e^{-(y\lambda/\theta)}dy\\
			&\leq \theta^{-1}C\int_{0}^{\theta}(y/\theta)^{-(\beta(\Pi)+\eps)}dy+\theta^{-1}\bar\Pi(1)\frac{\theta}{\lambda}\\
			&=C+C\frac{1}{\lambda},
		\end{align*}
		where we again used Lemma \ref{l4} but now $y/\theta$ does not explode for small $\theta$ as we only integrate up to $\theta$. Having an estimate for positive $\theta$ we note that $Im(\mathcal L\bar\Pi(\lambda+i\theta))$ as a function of $\theta$ is odd to deduce that
		\begin{equation}\label{LEq6}
			|Im(\mathcal L\bar\Pi(\lambda+i\theta))|\leq \begin{cases}
				                                           C\frac{1}{\lambda}&:|\theta|\leq 1,\\
				                                           C|\theta|^{\beta(\Pi)+\eps-1}&:|\theta|>1.
				                                          \end{cases}
		\end{equation}
		Similarly, we estimate the real part
		\begin{align*}\label{LEq7}
			\nonumber& \big|Re(\mathcal L\bar\Pi(\lambda+i\theta))\big|\\
			&\nonumber=\left|\int_{0}^{\infty}\cos(\theta y)r_{\lambda}(y)dy\right|\\
			 \nonumber&=\left|\int_{0}^{\frac{\pi}{2}}\cos(y)r_{\lambda}(y\theta^{-1})dy+\theta^{-1}\sum_{k=1}^{\infty}\int_{(4k+1)\frac{\pi}{2}}^{(4k+3)\frac{\pi}{2}}(r_{\lambda}(y\theta^{-1})-r_{\lambda}((y+\pi)\theta^{-1}))\cos(y)dy\right|\\
 			&\nonumber\leq\int_{0}^{\frac{\pi}{2}}r_{\lambda}(y\theta^{-1})dy\\
 			&\leq C|\theta|^{\beta(\Pi)+\eps-1}
		\end{align*}
		for large $|\theta|$ and precisely as above for small $|\theta|$. This finishes the proof of the lemma.
		\end{proof}
	The upper bound can now be used to derive the necessary integrability of $g^{N,q}$.	
	\begin{lemma}\label{l434}
		For arbitrary integer $N$ larger than $0$ and any $\lambda>0$ we have
  		\begin{equation}\label{LEq3}
			\int_{-\infty}^{\infty}\left|g^{N,q}(\lambda+i\theta)\right|\,d\theta <\infty
 		\end{equation}
 		for $g^{N,q}$ defined in (\ref{g}).
	\end{lemma}
	
	\begin{proof}
		As we have already found a good upper bound for $\mathcal L \bar\Pi(s)$ in the previous lemma, it suffices to show that the denominator
		\begin{align*}
			p(\lambda+i\theta)=1+\mathcal L (\bar\Pi+q)(\lambda+i\theta)
		\end{align*}
		is bounded away from zero. In Lemma \ref{l11} we have shown that $p(\lambda+i\theta)$ has no zeros for $\lambda\geq 0$ and, hence, by continuity of $p$ it suffices to show that as $|\theta|$ tends to infinity $p(\lambda+i\theta)$ stays bounded away from zero. To this end it suffices to note that from Lemma \ref{l43}
		\begin{align*}
			\lim_{|\theta|\to \infty}\mathcal L (\bar\Pi+q)(\lambda+i\theta)=\lim_{|\theta|\to \infty}\Big(\mathcal L \bar\Pi(\lambda+i\theta)+\frac{q}{\lambda+i\theta}\Big)=0.
		\end{align*}
		 Using the fact $|\mathcal L \1(\lambda+i\theta)|=|1/(\lambda+i\theta)|\leq \min\big\{\frac{1}{\lambda},\frac{1}{|\theta|}\big\}$ for $\lambda>0$, we employ Lemma \ref{l43} to obtain the upper bound
		\begin{align}
			\left|g^{N,q}(\lambda+i\theta)\right|&\leq C|\mathcal L \1(\lambda+i\theta)||\mathcal L (\bar\Pi+q)(\lambda+i\theta)|^N\nonumber\\
			&=C\left|\frac{1}{\lambda+i\theta}\right|\left|\mathcal L \bar\Pi(\lambda+i\theta)+\frac{q}{\lambda+i\theta}\right|^N
			\leq  \begin{cases}
					C'\frac{1}{\lambda^{N+1}}&:|\theta|\leq 1,\\
       					C'|\theta|^{N(\beta(\Pi)+\eps-1)-1}&:|\theta|>1.
			             \end{cases}\label{he}
		\end{align}
		The right hand side is integrable in $\theta$ as by assumption $\beta(\Pi)<1$ and $\eps$ can be chosen sufficiently small so that $N(\beta(\Pi)+\eps-1)<0$.
	\end{proof}

	We are now in a position to derive the Laplace inversion representations for $u^{(q)}$.
	
	\begin{proof}[Proof of Proposition \ref{laplaceinversionu}:]
		For $\lambda\geq \lambda_0$, $\lambda_0$ satisfying $\int_0^{\infty}e^{-\lambda_0 x}(\bar\Pi+q)(x)\,dx<1$ we can directly follow the strategy explained before Lemma \ref{l33}. The proof then follows directly from the definition of $\phi^{N,q}$ and Laplace inversion justified by Lemmas \ref{l33}, \ref{l11}, and \ref{l434}.\\
		The proof of the proposition is complete if we can show that for arbitrary $0<\lambda<\lambda_0$
		\begin{align*}
			\int_{\Gamma(\lambda)}e^{s x}g^{N,q}(s)\,ds=\int_{\Gamma(\lambda_0)}e^{s x}g^{N,q}(s)\,ds,
		\end{align*}
		where $\Gamma(\lambda)=\{\lambda+i\theta: \theta\in\R\}$. Laplace transforms are analytic (see for instance Theorem 75.2 of \cite{K88}) and $g^{N,q}$ has no singularity for $\lambda>0$ by Lemma \ref{l11}, hence, Cauchy's theorem applied to the closed contour formed by the pieces
		\begin{align*}
			\Gamma(\lambda_{0})&\cap\{|\theta|\leq R\},\\
			\Gamma(\lambda)&\cap\{|\theta|\leq R\},\\
			\Phi(R)&=\big\{s:\,s=r+iR, r\in [\lambda,\lambda_0]\big\},\\
			\tilde{\Phi}(R)&=\big\{s:\,s=r-iR, r\in [\lambda,\lambda_0]\big\},
		\end{align*}
		taken with the right orientation implies the claim. Note that the integrals over the horizontal pieces vanish as $R$ tends to infinity:
		\begin{align*}
			\lim_{R\to\infty}\Big|\int_{\Phi(R)}e^{sx}g^{N,q}(s)ds\Big|\leq Ce^{\lambda_0 x}(\lambda_{0}-\lambda)\lim_{R\to\infty}|R|^{N(\beta(\Pi)+\eps-1)-1},
		\end{align*}
		where we have used (\ref{he}) for $|\theta|>1$ to estimate $|g^{N,q}(s)|$ for $R$ big enough. Choosing $\eps$ small enough so that $\beta(\Pi)+\eps-1<0$, the right hand side tends to zero. The same argument shows that 		the integral over $\tilde{\Phi}(R)$ vanishes.
	\end{proof}

	\begin{proof}[Proof of Corollary \ref{t2}:]
		As remarked after the corollary, the arguments of \cite{CKS10} will not be repeated. Instead, under Assumption (\textbf{A}) we prove the Laplace inversion representation and deduce from this the series representation of \cite{CKS10}.\\
		We first show that the right and left derivatives of $u^{(q)}$ exist and are given by the representation of the theorem. First, right and left derivatives of the finite sum in (\ref{invers}) exist by termwise differentiating the finite sum and 		using that $\1\ast \big(\bar\Pi+q\big)^{\ast n}(x)=\int_0^x\big(\bar\Pi+q\big)^{\ast n}(y)\,dy$ is differentiable from the left and the right with derivative $\big(\bar\Pi+q\big)^{\ast n}(x-)$ (resp. $\big(\bar\Pi+q\big)^{\ast n}(x+)$). As 		iterated convolutions are continuous, only the first summand is not everywhere differentiable.\\		
		To see that the integral is differentiable at $x$ and to deduce the integral representation of $({u^{(q)}})'$ note that
		\begin{align*}
			\frac{d}{dx}e^{\lambda x}\frac{1}{2\pi}\int_{-\infty}^{\infty}e^{i\theta x}g^{N,q}(\lambda+i\theta)\,d\theta
			&=e^{\lambda x}\frac{1}{2\pi}\int_{-\infty}^{\infty}(\lambda+i\theta)e^{i\theta x}g^{N,q}(\lambda+i\theta)\,d\theta\\
			&=e^{\lambda x}\frac{1}{2\pi}\int_{-\infty}^{\infty}e^{i\theta x}h^{N,q}(\lambda+i\theta)\,d\theta.
		\end{align*}
		The differentiation under the integral is justified by dominated convergence and the upper bound
		\begin{align*}
			\left|\frac{d}{dx}e^{i\theta x}g^{N,q}(\lambda+i\theta)\right|=\left|\theta g^{N,q}(\lambda+i\theta)\right|\leq  \begin{cases}
 					C|\theta|\frac{1}{\lambda^{N+1}}&:|\theta|\leq 1,\\
					C|\theta|^{N(\beta(\Pi)+\eps-1)}&:|\theta|>1,\\
			             \end{cases}	
		\end{align*}
		derived in (\ref{he}) which is integrable in $\theta$ for sufficiently small $\eps$ by our choice of $N$.\medskip
		
		As a second step we now derive the pointwise series representation from the Laplace transform representation of $({u^{(q)}})'$. As for (\ref{he}) we obtain the upper bound
		\begin{align}\label{he2}
			\left|h^{N,q}(\lambda+i\theta)\right|\leq C|\mathcal L \bar\Pi(\lambda+i\theta)|^N\leq  \begin{cases}
					C\frac{1}{\lambda^{N}}&:|\theta|\leq 1,\\
       					C|\theta|^{N(\beta(\Pi)+\eps-1)}&:|\theta|>1,
			             \end{cases}		
		\end{align}
		for arbitrary $\eps>0$.
		To prove the corollary it suffices to show that for $N$ tending to infinity, the Laplace inversion integral
		\begin{align*}
			e^{\lambda x}\frac{1}{2\pi}\int_{-\infty}^{\infty}e^{i\theta x}h^{N,q}(\lambda+i\theta)\,d\theta
		\end{align*}
		vanishes for fixed $x>0$ and $\lambda>0$. With our choice of $\varepsilon$, i.e. $\beta(\Pi)+\eps-1<0$,  (\ref{he2}) implies pointwise convergence $e^{i\theta x}h^{N,q}(\lambda+i\theta)\to 0$. We are done if we can justify the change of limit and integration. This comes from the uniform (for $N\geq 2[(\beta(\Pi)+\eps-1)^{-1}]+2$ ) in $\theta$ integrable upper bound
		\begin{align*}
			\left|e^{i\theta x}h^{N,q}(\lambda+i\theta)\right|\leq  \begin{cases}
					C&:|\theta|\leq 1,\\
       					C|\theta|^{-2}&:|\theta|>1,
			             \end{cases}		
		\end{align*}
		and the dominated convergence theorem.		
	\end{proof}


\subsection{Higher Order (Non)Differentiability}
	In this section the results on differentiability are proved. In contrast to Corollary \ref{uu}, which follows either from differentiating the Laplace inversion representation or differentiating termwise the series representation of $u^{(q)}$, the proofs for higher order derivatives are exclusively based on the more elegant Laplace inversion approach. This forces us to assume $\beta(\Pi)<1$ and we do not see how to circumvent this (probably dispensable) restriction.\\
	To reduce the proofs to (non)differentiability of iterated convolutions, differentiability of the Laplace inversion integral is ensured in the next lemma if $N$ is sufficiently large.
	
	\begin{lemma}\label{l111}
		For $N>-\frac{k}{\beta(\Pi)-1}$, the Laplace inversion integral in (\ref{invers}) is everywhere $k$-times continuously differentiable in $x$ with
		\begin{align*}
			\frac{d^k}{dx^k}e^{\lambda x}\frac{1}{2\pi}\int_{-\infty}^{\infty}e^{i\theta x}g^{N,q}(\lambda+i\theta)\,d\theta=\frac{1}{2\pi}\int_{-\infty}^{\infty}e^{(\lambda+i\theta) x}(\lambda+i\theta)^{k}g^{N,q}(\lambda+i\theta)\,d\theta.
		\end{align*}
	\end{lemma}
	\begin{proof}
		To check that we can differentiate
		\begin{align*}
			e^{\lambda x}\frac{1}{2\pi}\int_{-\infty}^{\infty}e^{i\theta x}g^{N,q}(\lambda+i\theta)\,d\theta
		\end{align*}
		$k$-times under the integral, by dominated convergence we need to find an integrable upper bound for the derivative:
		\begin{align*}
			\Big|\frac{d^k}{dx^k}e^{i\theta x}g^{N,q}(\lambda+i\theta)\Big|=\Big|\theta^kg^{N,q}(\lambda+i\theta)\Big|.
		\end{align*}
		Integrability follows directly from the choice of $N$ and the integrable upper bound (\ref{he}) for sufficiently small $\eps$.
	\end{proof}
	
	Having understood differentiability of the remainder term, to prove higher order differentiability of $u^{(q)}$ we choose $N$ large enough to apply the previous lemma and then deal with the finite sum of convolutions. Here is a lemma for smooth L\'evy measures.
	\begin{lemma}\label{lem:dif}
		If $f:\R\to\R^+$ with $f(x)=0$ for $x\leq 0$ is infinitely differentiable on $\R^+$ and locally integrable at zero, then $f^{\ast n}$ is everywhere infinitely differentiable and integrable at zero for any $n\geq 1$.
	\end{lemma}
	\begin{proof}
	The proof is easily conducted by induction with basis $n=1$, i.e. $f(x)$. Then the simple identity
	\begin{align*}
		f^{\ast (n+1)}(x)=\int_{0}^{\frac{x}{2}}f^{\ast n}(x-y)f(y)dy+\int_{0}^{\frac{x}{2}}f^{\ast n}(y)f(x-y)dy
	\end{align*}
	and the induction hypothesis confirm the statement of the lemma with respect to differentiability. The integrability follows similarly from the representation
	\[\int_{0}^{1}f^{*(n+1)}(x)dx=\int_{0}^{1}f^{*n}(y)\int_{0}^{1-y}f(s)dsdy\]
	and the integrability of $f(x)$ at zero.
	\end{proof}

	Combining the previous lemmas we can prove Theorem \ref{t12}.

	\begin{proof}[Proof of Theorem \ref{t12}:]
		As the potential measure $U^{(q)}$ is differentiable with derivative $u^{(q)}$, to differentiate $(k+1)$-times $U^{(q)}$ it suffices to differentiate $k$-times the potential density $u^{(q)}$. Applying Proposition \ref{laplaceinversionu} with $N>-\frac{k}{\beta(\Pi)-1}$ yields
 		\begin{align*}
			\frac{d^{k+1}}{dx^{k+1}}U^{(q)}(x)&=\frac{d^k}{dx^k}u^{(q)}(x)\\
			&=\frac{d^{k}}{dx^{k}}\sum_{n=1}^{N-1}\frac{(-1)^n}{\delta^{n+1}}\1\ast(\bar\Pi+q)^{\ast n}(x)
			+\frac{d^{k}}{dx^{k}}e^{\lambda x}\frac{1}{2\pi}\int_{-\infty}^{\infty}e^{i\theta x}g^{N,q}(\lambda+i\theta)\,d\theta.
		\end{align*}
		As $N$ was assumed to be large enough, the integral is everywhere $k$-times differentiable by Lemma \ref{l111}. For $\beta(\Pi)$ close to $1$ the sum might be arbitrarily large but is always finite. So we can differentiate termwise the sum and use $\frac{d}{dx}\1\ast(\bar\Pi+q)^{\ast n}(x)=(\bar\Pi+q)^{\ast n}(x)$ as $\bar\Pi$ is continuous leading to
		\begin{align}\label{F}
			 \frac{d^{k+1}}{dx^{k+1}}U^{(q)}(x)=\sum_{n=1}^{N-1}\frac{(-1)^n}{\delta^{n+1}}\frac{d^{k-1}}{dx^{k-1}}(\bar\Pi+q)^{\ast n}(x)
			+e^{\lambda x}\frac{1}{2\pi}\int_{-\infty}^{\infty}e^{i\theta x}(\lambda+i\theta)^{k}g^{N,q}(\lambda+i\theta)\,d\theta.
		\end{align}
		Applying Lemma \ref{lem:dif} to each summand of the finite sum concludes the proof.
	\end{proof}
	
	The proof revealed the full strength of the Laplace inversion approach compared to the series approach of \cite{CKS10}. Their major technical problems consist of justifying differentiation under the alternating infinite sum (\ref{series}). As we split the infinite sum into a harmless finite sum and an integral which can be dealt with easily, the main problems of \cite{CKS10} have been circumvented.\medskip
		
	Next we analyze the influence of atoms of $\Pi$ on (non)differentiability for which we start with a lemma on higher order convolutions for discrete L\'evy measures. 
		  %
	
	\begin{lemma}\label{l}
		Suppose $\Pi$ is purely atomic with possible accumulation point of atoms only at zero, then
		\begin{itemize}
			\item[(a)] $\bar\Pi^{\ast n}$ is a polynomial of order at most $n-1$ away from $G_n$ for $n\geq 1$,
			\item[(b)] $\bar\Pi^{\ast n}$ is everywhere $(n-2)$-times differentiable but not $(n-1)$-times differentiable at $G_n\backslash G_{n-1}$ for $n\geq 2$.
		\end{itemize}
	\end{lemma}
	\begin{proof}
		The proof is based on the simple observation
		\begin{align}
			\1_{[0,a]}\ast f(x)=\begin{cases}
						\int_{x-a}^xf(y)\,dy&:a<x,\\
			                     	\int_0^xf(y)\,dy&:a\geq x,
			                    \end{cases}\label{mm}
		\end{align}
		for integrable $f$ vanishing on the negative half-line. Hence, $\1_{[0,a]}\ast f$ is continuous everywhere and differentiable at $x$ and $x+a$ if and only if $f$ is continuous at $x$.
		We will resort to the fact that as $\Pi$ is purely atomic, $\bar\Pi=\Pi(x,\infty)$ is piecewise constant and can be represented as linear combination of step functions, i.e.
		\begin{align}
			\bar\Pi(x)=\sum_{a\in G} \Pi(\{a\})\mathbf{1}_{[0,a]}(x).\label{ho}
		\end{align}
		To prove (a) we proceed by induction which we start with $n=1$. Taking into account (\ref{ho}), $\bar\Pi^{\ast 1}=\bar\Pi$ is a polynomial of order $0$ away from $G$. Next, assume that $\bar\Pi^{\ast n}$ is a polynomial of order at most $n-1$ away from $G_n$.
		To prove the claim for $\bar\Pi^{\ast(n+1)}$, we fix an interval $(d,b)$ such that $d\in G_{n+1}\cup\{0\}$, $b\in G_{n+1}$ and $(d,b)\cap G_{n+1}=\emptyset$. Our aim is to show that for every $x\in (d,b)$ there is an open interval $\Delta(x)\subset (d,b)$ such that $\bar\Pi^{\ast(n+1)}$ is polynomial of order at most $n$ on $\Delta(x)$. Fix $x\in (d,b)$ and choose $a^{*}(x)=\max\big\{c\in G\cup \{0\}\,:\,x-c>d\big\}$. As the atoms accumulate only at zero there is a neighbourhood $\Delta(x)$ of $x$ such that $z-a^{*}(x)>d$ for every $z\in \Delta(x)$. Next, by Fubini's theorem we observe
		\begin{align}\label{hie}
			\bar\Pi^{\ast (n+1)}(x)=\sum_{a\in G}\Pi(\{a\})\1_{[0,a]}\ast \bar\Pi^{\ast n}(x).
		\end{align}
		By the induction hypothesis and (\ref{mm}) all summands are locally polynomials of order at most $n$. To show that $\bar\Pi^{\ast (n+1)}$ is a polynomial of order at most $n$, we split $\bar\Pi^{\ast n}$ according to the two cases of (\ref{mm}):
		\begin{align*}
		\bar\Pi^{\ast (n+1)}(x)&=\sum_{x\leq a\in G}\Pi(\{a\})\1_{[0,a]}\ast \bar\Pi^{\ast n}(x)+\sum_{x>a\in G}\Pi(\{a\})\1_{[0,a]}\ast \bar\Pi^{\ast n}(x)\\
		&=\int_0^x\bar\Pi^{\ast n}(y)\,dy\sum_{x\leq a\in G}\Pi(\{a\})+\sum_{x>a\in G}\Pi(\{a\})\int_{x-a}^x\bar\Pi^{\ast n}(y)\,dy.
		\end{align*}
		The first summand clearly is a polynomial of order at most $n$ on $\Delta(x)$ by the induction hypothesis using that $\sum_{a\geq x}\Pi(\{a\})$ is a constant on $\Delta(x)$. To show that the second summand is a polynomial as well, we write for all $z\in \Delta(x)$:
		\begin{align*}
		&\sum_{z>a\in G}\Pi(\{a\})\int_{z-a}^z\bar\Pi^{\ast n}(y)\,dy\\
		&=\sum_{a^{*}(x)\geq a\in G}\Pi(\{a\})\int_{z-a}^z\bar\Pi^{\ast n}(y)\,dy+\sum_{a^{*}(x)<a<z, a\in G}\Pi(\{a\})\int_{z-a}^z\bar\Pi^{\ast n}(y)\,dy.
		\end{align*}
		We clearly deduce by the induction hypothesis that the second sum, having finitely many summands only, is a polynomial of order at most $n$ on $\Delta(x)$. By the definition of $a^{*}(x)$ and the induction hypothesis, $\bar\Pi^{\ast n}$ is the same polynomial on $(d,b)$. Also, since for $z-a>d$
		\begin{align*}
			\int_{z-a}^z\bar\Pi^{\ast n}(y)\,dy\leq a \sup_{s\in (d,b)}|\bar\Pi^{\ast n}(s)|\leq Ca,
		\end{align*}
		the first sum is clearly absolutely summable and henceforth it defines a polynomial of order at most $n$. Thus we conclude that on $\Delta(x)$ we have that
		$\bar\Pi^{\ast (n+1)}$ is a polynomial of order at most $n$. Representing $(b,d)$ as a union of neighbourhoods $\Delta(x)$ shows that $\bar\Pi^{\ast (n+1)}$ is indeed a polynomial of order at most $n$ on $(b,d)$.\\
	
		In particular, (a) shows that $\bar\Pi^{\ast n}$ is infinitely differentiable away from $G_n$. To prove the claimed non-differentiability in (b), a different approach is needed. We prove the assertion again by induction in $n$. The first step is to show that $\bar\Pi^{\ast 2}$ is everywhere continuous and not differentiable at $G_2\backslash G$. Continuity follows from the continuity of $\1_{[0,a]}\ast \bar\Pi(x)$ and the locally uniform convergence of (\ref{hie}) which can, applying monotonicity of $\bar\Pi$ and the properties of $\Pi$, be seen for $\eps$ small enough and $N$ large enough from
		\begin{align*}
			&\sup_{x}\left(\sum_{\eps>a\in G}\Pi(\{a\})\1_{[0,a]}\ast \bar\Pi(x)+\sum_{N<a\in G}\Pi(\{a\})\1_{[0,a]}\ast \bar\Pi(x)\right)\\
			&=\sup_{x}\left(\sum_{\eps>a\in G}\Pi(\{a\})\int_{x-a}^x\bar\Pi(y)\,dy+\sum_{N<a\in G}\Pi(\{a\})\int_0^x\bar\Pi(y)\,dy\right)\\
			&\leq\sup_{x}\left(\sum_{\eps>a\in G}\Pi(\{a\})a\bar\Pi(x-\eps)+\int_0^x\bar\Pi(y)\,dy\sum_{N<a\in G}\Pi(\{a\})\right)\\
			&\leq\sup_xC(x)\left(\sum_{\eps>a\in G}\Pi(\{a\})a+\sum_{N<a\in G}\Pi(\{a\})\right)\\
			&\stackrel{\eps\to 0, N\to \infty}{\longrightarrow} 0,
		\end{align*}
		where $C(x)$ are locally bounded constants.\\
		Now we show that $\bar\Pi^{\ast 2}$ is not differentiable at $G_2\backslash G$. Although we already know that $\bar\Pi^{\ast 2}$ is a polynomial away from $G_2$, we derive a second representation showing that for $x\notin G_2$ it can be differentiated termwise. As the atoms are discrete, there is a largest atom $c<x$ implying that $G\cap (c,x]=\emptyset$. Taking into account (\ref{mm}), that $\bar\Pi$ is constant in $(x-c,x]$, and the definition of $G_2$, we see that $\1_{[0,a]}\ast \bar\Pi$ is infinitely differentiable away from $G_2$. To show that the infinite sum (\ref{hie}) can be differentiated termwise for $n=1$, we derive a locally absolutely and uniformly summable in $x$ upper bound for the series of derivatives:
		\begin{align*}
			\sum_{a\in G} \Pi(\{a\})\left|\frac{d}{dx}\mathbf{1}_{[0,a]}\ast \bar\Pi(x)\right|
			&=\sum_{x\leq a\in G} \Pi(\{a\})\bar\Pi(x)+\sum_{x>a\in G} \Pi(\{a\})\left|\bar\Pi(x)-\bar\Pi(x-a)\right|.
		\end{align*}
		The first summand is bounded as $\bar\Pi$ is locally constant. For the second, the sum only runs over $x-c<a$ as otherwise $\bar\Pi(x)-\bar\Pi(x-a)=0$. There is no accumulation point of atoms at $x$, so the second summand is bounded by $C\Pi([x-c,x])<\infty$ and, hence, $\bar\Pi^{\ast 2}$ can be differentiated termwise.\\
		With this in hand we can show non-differentiability at $G_{2}\backslash G$. Let $b\in G_{2}\backslash G$, then, since $G$ has a possible accumulation point only at $0$, there is a neighbourhood $A(b)$ of $b$ such that $[y,x]\cap G_{2}=\{b\}$ for $y, x\in A(b)$. Since $b\in G_{2}\backslash G$ we have $\bar\Pi(z)=\bar\Pi(b)$ for all $z\in A(b)$ so that
		\begin{align*}
			&\big(\bar\Pi^{\ast 2}\big)'(b+)-\big(\bar\Pi^{\ast 2}\big)'(b-)\\
			&=\lim_{x\downarrow b,\,y\uparrow b}\left[\big(\bar\Pi^{\ast 2}\big)'(x)-\big(\bar\Pi^{\ast 2}\big)'(y)\right]\\
			&=\lim_{x\downarrow b,\,y\uparrow b}\Big(\sum_{b\leq a\in G} \Pi(\{a\})(\bar\Pi(x)-\bar\Pi(y))
			+\sum_{b>a\in G} \Pi(\{a\})\left(\bar\Pi(x)-\bar\Pi(x-a)-\bar\Pi(y)+\bar\Pi(y-a)\right)\Big)\\
			&=\lim_{x\downarrow b,\,y\uparrow b}\sum_{b>a\in G}\Pi(\{a\})(\bar\Pi(y-a)-\bar\Pi(x-a))\\
			&=\sum_{b>a\in G}\Pi(\{a\})\Pi(\{b-a\}).
		\end{align*}
		Since $b\in G_{2}\backslash G$ and $G$ has a possible accumulation point only at zero, the last sum has a finite number of summands and is strictly positive. The exchange of limit and summation is possible since for all $y$ and $x$ sufficiently close to $b$ and all $a\in G$, such that $a<c$, for some $c>0$, $y-a\in A(b)$ and $x-a\in A(b)$. The latter implies that $\bar\Pi(x-a)=\bar\Pi(y-a)$ and the sum in the limit is a finite sum.  This proves (b) for $n=2$.\medskip
		
		Next, assume that $\bar\Pi^{\ast n}$ is everywhere $(n-2)$-times differentiable and that for $b\in G_{n}\backslash G_{n-1}$
		\begin{align*}
			0<\frac{d^{n-1}}{dx^{n-1}}\bar\Pi^{\ast n}(b+)-\frac{d^{n-1}}{dx^{n-1}}\bar\Pi^{\ast n}(b-)<\infty.
		\end{align*}
		To show that at the critical points $\bar\Pi^{\ast (n+1)}$ is everywhere $(n-1)$-times differentiable we again verify termwise the differentiability:
		\begin{align*}
			&\sum_{a\in G}\Pi(\{a\})\left|\frac{d^{n-1}}{dx^{n-1}}\1_{[0,a]}\ast \bar\Pi^{\ast n}(x)\right|\\
			&=\sum_{x>a\in G}\Pi(\{a\})\left|\frac{d^{n-2}}{dx^{n-2}}\big(\bar\Pi^{\ast n}(x)-\bar\Pi^{\ast n}(x-a)\big)\right|+\sum_{x\leq a\in G}\Pi(\{a\})\left|\frac{d^{n-2}}{dx^{n-2}}\bar\Pi^{\ast n}(x)\right|.
		\end{align*}
		As $\bar\Pi^{\ast n}$ is everywhere $(n-2)$-times continuously differentiable, the second sum is locally bounded by the property (\ref{c}). Clearly for any $x\in G_n\backslash G_{n-1}$ there is an interval $(x-d,x)$ such that $(x-d,x)\cap G_{n}=\emptyset$ and $d<x$ because $G$ has a possible accumulation point only at zero. The latter also implies that there are at most finitely many atoms $a$ such that $x>a\geq d$ and therefore we need to study only
		\begin{align*}
			\sum_{d>a\in G}\Pi(\{a\})\left|\frac{d^{n-2}}{dx^{n-2}}\big(\bar\Pi^{\ast n}(x)-\bar\Pi^{\ast n}(x-a)\big)\right|.
		\end{align*}
		But from (a), $\bar\Pi^{\ast n}(y)$ is a polynomial of order at most $n-1$ for all $y\in(x-d,x)$ so that by the mean value theorem
		\begin{align}\label{le}
			\left|\frac{d^{n-2}}{dx^{n-2}}\bar\Pi^{\ast n}(x)-\frac{d^{n-2}}{dx^{n-2}}\bar\Pi^{\ast n}(x-a)\right|\leq Ca.
		\end{align}
		As $\sum_{a<d}a\Pi(\{a\})<\infty$ we can interchange differentiation and summation and then using the induction hypothesis that $\bar\Pi^{\ast n}$ is $(n-2)$-times continuously differentiable everywhere. In total we conclude that $\bar\Pi^{\ast (n+1)}$ is $(n-1)$-times continuously differentiable everywhere. \medskip

	Finally our task is to show that for $b\in G_{n+1}\backslash G_{n}$
  	\begin{align*}
  		0<\frac{d^{n}}{dx^{n}}\bar\Pi^{\ast (n+1)}(b+)-\frac{d^{n}}{dx^{n}}\bar\Pi^{\ast (n+1)}(b-)<\infty.
  	\end{align*}
  	As $b\in G_{n+1}\backslash G_{n}$ and there is a neighbourhood $A(b)$ of $b$ such that $A(b)\cap{G_{n+1}}=\{b\}$, the arguments above imply that we can differentiate termwise $n$-times \eqref{hie} on $A(b)\backslash b$ to get
  	\begin{align*}
  		\frac{d^{n}}{dx^{n}}\bar\Pi^{\ast (n+1)}(x)=\sum_{x-a\not\in A(b),a\in G}\Pi(\{a\})\frac{d^{n}}{dx^{n}}1_{[0,a]}\ast\bar\Pi^{\ast n}(x)+\sum_{x-a\in A(b),a\in G}\Pi(\{a\})\frac{d^{n}}{dx^{n}}1_{[0,a]}\ast\bar\Pi^{\ast n}(x).
  	\end{align*}
  	Using \eqref{mm} we rewrite the latter as
  	\begin{align*}
  		&\frac{d^{n}}{dx^{n}}\bar\Pi^{\ast (n+1)}(x)\\
  		&=\sum_{x\leq a\in G}\Pi(\{a\})\frac{d^{n-1}}{dx^{n-1}}\bar\Pi^{\ast n}(x)+\sum_{x> a\in G}\Pi(\{a\})\left(\frac{d^{n-1}}{dx^{n-1}}\bar\Pi^{\ast n}(x)-\frac{d^{n-1}}{dx^{n-1}}\bar\Pi^{\ast n}(x-a)\right).
  	\end{align*}
  	Now, as $b\in G_{n+1}\backslash G_{n}$, $\bar\Pi^{\ast n}(x)$ is continuously differentiable on $A(b)$, $\Pi(\{b\})=0$ and on $A(b)$ $\frac{d^{n-1}}{dx^{n-1}}\bar\Pi^{\ast n}(x)=\frac{d^{n-1}}{dx^{n-1}}\bar\Pi^{\ast n}(y)$, because according to Lemma \ref{l}, $\bar\Pi^{\ast n}(x)$ is a polynomial of order at most $n-1$ on $A(b)$, we obtain
  	\begin{align*}
  	&\frac{d^{n}}{dx^{n}}\bar\Pi^{\ast (n+1)}(b+)-\frac{d^{n}}{dx^{n}}\bar\Pi^{\ast (n+1)}(b-)\\
  		&=\lim_{x\downarrow b,\,y\uparrow b}\left(\frac{d^{n}}{dx^{n}}\bar\Pi^{\ast (n+1)}(x)-\frac{d^{n}}{dx^{n}}\bar\Pi^{\ast (n+1)}(y)\right)\\
  		&=\lim_{x\downarrow b,\,y\uparrow b}\sum_{x> a\in G}\left(\frac{d^{n-1}}{dx^{n-1}}\bar\Pi^{\ast n}(x-a)-\frac{d^{n-1}}{dx^{n-1}}\bar\Pi^{\ast n)}(y-a)\right)\\
  		&=\sum_{b> a\in G}\Pi(\{a\})\left(\frac{d^{n-1}}{dx^{n-1}}\bar\Pi^{\ast n}((b-a)+)-\frac{d^{n-1}}{dx^{n-1}}\bar\Pi^{\ast n}((b-a)-)\right).
  	\end{align*}
  	Since $b\in G_{n+1}\backslash G_{n}$ we have that either $b-a\not\in G_{n}$ or $b-a\in G_{n}\backslash G_{n-1}$, where the latter is possible only for finitely many, but more than zero, $a\in G$. We mention that the interchange of limit and summation is valid since for all $x$ and $y$ sufficiently close to $b$, there is $c>0$, such that for all $a\in G$ such that $a<c$, $x-a\in A(b)$ and $y-a\in A(b)$. We have already argued above that the $n-1$-st derivative of $\bar\Pi^{\ast n}$ is a constant on $A(b)$ due to Lemma \ref{l}. By the induction hypothesis for those $a$
  	\begin{align*}
  		\frac{d^{n-1}}{dx^{n-1}}\bar\Pi^{\ast n}((b-a)+)-\frac{d^{n-1}}{dx^{n-1}}\bar\Pi^{\ast n}((b-a)-)>0.
  	\end{align*}	
  	Thus we conclude the induction and the proof of the lemma.
	\end{proof}

	The observation of the lemma motivates the strategy for the main proofs: first, choose $N$ large enough that the integral of the Laplace inversion representation can be differentiated as often as needed. Secondly, consider the finite sum of iterated convolutions for which critical points for $k$th derivatives only occur for the $k$th summand. Lower order convolutions vanish and higher order convolutions are smooth enough.\smallskip
	 	
	Before we give the main proofs, one more lemma is needed to show how to separate $\bar\Pi$ and $q$ and afterwards absolutely continuous and discrete part of the L\'evy measure.
	
	\begin{lemma}\label{pre}
		Suppose $\Pi$ is purely atomic with possible accumulation point only at zero. If $f$ is infinitely differentiable away from $G_i$ with locally bounded derivatives and integrable at zero, then
		$\bar\Pi\ast f$ is infinitely differentiable away from $G_{i+1}$ with locally bounded derivatives.
	\end{lemma}
	\begin{proof}
		First note that by (\ref{mm}) for arbitrary $a>0$
		\begin{align}
			\frac{d^k}{dx^k}\mathbf{1}_{[0,a]}\ast f(x)=\begin{cases}
						\frac{d^{k-1}}{dx^{k-1}}f(x)-\frac{d^{k-1}}{dx^{k-1}}f(x-a)&:a<x,\\
			                                            \frac{d^{k-1}}{dx^{k-1}}f(x)&:a\geq x,
			                                             \end{cases}\label{e}
		\end{align}
		From this simple identity the claim follows for the special case $\Pi=\delta_a$. For infinitely many atoms the difficulty appears from the fact that $\bar\Pi$ is an infinite sum of indicator functions so that summation and differentiation in
		\begin{align}\label{xy}
			\bar\Pi\ast f(x)&=\left(\sum_{a\in G}\Pi(\{a\})\mathbf{1}_{[0,a]}\right)\ast f(x)=\sum_{a\in G}\Pi(\{a\})\left(\mathbf{1}_{[0,a]}\ast f\right)(x)
		\end{align}
		need to be interchanged. To justify the Fubini flip in (\ref{xy}) it suffices to show locally uniform absolute convergence of the right hand side. An upper bound can be obtained as
		\begin{align*}
			&\quad\sum_{x>a\in G}\Pi(\{a\})\int_{x-a}^x|f(y)|\,dy+\sum_{x\leq a\in G}\Pi(\{a\})\int_0^x|f(y)|\,dy\\
			&\leq\sup_{y\in [x-c,x]}|f(y)|\sum_{x>a\in G}\Pi(\{a\})a+\int_0^x|f(y)|\,dy\sum_{x\leq a\in G}\Pi(\{a\}),
		\end{align*}
		where $c$ denotes the largest atom strictly smaller than $x$ which exists as the atoms are discrete away from zero. The right hand side is finite by property (\ref{c}), continuity of $f$, and integrability of $f$ at zero.\\
		Having justified (\ref{xy}), we now show that away from $G_{i+1}$
	 	\begin{align}\label{MMM1}
	 		\frac{d^{k}}{dx^{k}}\bar\Pi\ast f(x)=\sum_{a\in G}\Pi(\{a\})\frac{d^{k}}{dx^{k}}\left(\mathbf{1}_{[0,a]}\ast f\right)(x)
	 	\end{align}
	 	which we prove by showing locally uniform absolute convergence of the series of derivatives. First note that as $x\notin G_{i+1}$ there is $c'>0$ such that $(x-c',x+c')\cap G_{i+1}=\varnothing$. As the set $B=\{a\in G: a>c'\}$ is finite, we appeal to (\ref{e}) and the mean value theorem, we obtain for $a\in B^{c}\cap G$
	 	\begin{align*}
			\frac{d^k}{dx^k}\mathbf{1}_{[0,a]}\ast f(x)\leq
				\sup_{y\in [x-c',x]\cap G}\Big|\left(\frac{d^{k}}{dy^{k}}f(y)\right)a\Big|.&
		\end{align*}
		 Then
		\begin{align*}
			&\quad\sum_{a\in B^{c}\cap G}\Pi(\{a\})\left|\frac{d^k}{dx^k}(\mathbf{1}_{[0,a]}\ast f)(x)\right|\\
			&\leq \sup_{y\in [x-c,x]}\left|\frac{d^{k}}{dy^{k}}f(y)\right|\sum_{a\in B^{c}\cap G
}\Pi(\{a\})a,
		\end{align*}
		which again is bounded by property (\ref{c}) and local boundedness of derivatives of $f$. This is enought to show \eqref{MMM1} as $B$ is a finite set. \\
		In total we proved that $\bar\Pi\ast f$ is infinitely differentiable away from $G_{i+1}$ and derivatives may be taken termwise. Local boundedness of the derivatives follows from the above estimate.
	\end{proof}

	With the lemmas in mind we can give the proofs of the main results.
	
	\begin{proof}[Proof of Theorem \ref{t4}]	
		To prove higher order (non)differentiability, we consider the Laplace inversion representation of Proposition \ref{laplaceinversionu} for $N>-\frac{k}{\beta(\Pi)-1}$:
		\begin{align}
			 \frac{d^k}{dx^k}u^{(q)}(x)=\frac{d^{k}}{dx^{k}}\sum_{n=1}^{N-1}\frac{(-1)^n}{\delta^{n+1}}\1\ast(\bar\Pi+q)^{\ast n}(x)+\frac{d^{k}}{dx^{k}}e^{\lambda x}\frac{1}{2\pi}\int_{-\infty}^{\infty}e^{i\theta x}g^{N,q}(\lambda+i\theta)\,d\theta.\label{be}
		\end{align}
		As $N$ was assumed to be large enough, the integral is everywhere $k$-times differentiable by Lemma \ref{l111} so that the critical points claimed in the theorem have to come from differentiating the finite (possibly very large) sum. As motivated before the proof, $N$ can be chosen arbitrarily large as the higher order convolutions are smooth enough.\\
		Since representation \eqref{be} is valid replacing $k$ by some $n\leq k$, we proceed by induction showing that {$u^{(q)}$ is $k$-times differentiable in $x$ iff $x\not\in G_{k}$}. The induction basis for $n=1$ is true in complete generality (without Assumption (\textbf{A})) due to Corollary \ref{uu}. Assume next that the claim is true for some $n<k$ and consider \eqref{be} with $k$ replaced by $n+1$. Since the $n$th derivative does not exist on $G_{n}$, we only need to consider the $(n+1)$st derivative on $\mathbb{R}_{+}\backslash  G_{n}$. The integral term is $(n+1)$-times differentiable as seen above so that we only consider the sum
		\begin{align*}
			&\quad\sum_{i=1}^{N-1}\frac{(-1)^i}{\delta^{i+1}}\1\ast(\bar\Pi+q)^{\ast i}(x)\\
			&=\sum_{i=1}^{n}\frac{(-1)^i}{\delta^{i+1}}\1\ast (\bar\Pi+q)^{\ast i}(x)+\frac{(-1)^{n+1}}{\delta^{n+2}}\1\ast(\bar\Pi+q)^{\ast (n+1)}(x)+\sum_{i=n+2}^{N-1}\frac{(-1)^i}{\delta^{i+1}}\1\ast(\bar\Pi+q)^{\ast i}(x).
		\end{align*}
		We start with the case $q=0$. By Lemma \ref{l}, $\bar\Pi^{\ast i}(x)$ is everywhere $(i-2)$-times continuously differentiable implying everywhere the existence of $n+1$ continuous derivatives for the third summand. According to Lemma \ref{l} the first sum is everywhere sufficiently differentiable away from $G_{n}$. Therefore we are left to deal with the middle term. But again according to Lemma \ref{l}, $\bar\Pi^{\ast (n+1)}$ is $n-$times differentiable on $\mathbb{R}_{+}\backslash  G_{n+1}$ with jumps of the $(n+1)$st derivative on $G_{n+1}\backslash G_{n}$. This shows that the $(n+1)$st derivative of $u$ exists iff $x\notin G_{n+1}$. In particular, setting $k=n+1$ and using ${U}'=u$, the proof is complete for $q=0$.
		To extend the result to $q>0$, we use that by linearity
		\begin{align*}
			(\bar\Pi+q)^{\ast n}(x)
			&=\sum_{k=0}^n\left(n\atop k\right)\bar\Pi^{\ast k}\ast q^{\ast (n-k)}(x)\\
			&=\bar\Pi^{\ast n}(x)+\sum_{k=0}^{n-1}\left(n\atop k\right)\bar\Pi^{\ast k}\ast q^{\ast (n-k)}(x).
		\end{align*}
		Plugging this into the previous equation reduces the problem to the case $q=0$ and additional convolutions with higher order convolutions of the constant function $q$. By Lemma \ref{pre} and induction, smoothness of 			higher convolutions of the constant function $q$ implies that the convolutions $\bar\Pi^{\ast k}\ast q^{\ast (n-k)}$ are differentiable away from $G_k$. As $k\leq n-1$ this shows that the additional convolutions do not generate 	additional points of non-differentiability.\\
		Hence, that the claimed differentiability property for $u^{(q)}$ follows from that for $u$.
	\end{proof}

	\begin{proof}[Proof of Theorem \ref{t9}:]
		The strategy is similar to the one of Theorem \ref{t4}. Choosing again $N>-\frac{k}{\beta(\Pi)-1}$, representation (\ref{be}) holds by Proposition \ref{laplaceinversionu} and taking into account Lemma \ref{l111} implies that we 		 only need to discuss differentiability of $\bar\Pi^{\ast n}=(\bar\Pi_1+\bar\Pi_2+q)^{\ast n}$ for $n=1,...,N-1$. \\
		The first term $\bar\Pi_1+\bar\Pi_2+q$ is infinitely differentiable precisely for any $x\notin G$ as $\bar\Pi_1$ is constant away from the atoms and jumps downwards on $G$ and $\bar\Pi_2$ is infinitely differentiable. For the 		 iterated convolutions we have to separate the contributions of $\bar\Pi_1$ and $\bar\Pi_2$. Writing again
		\begin{align*}
			(\bar\Pi_1+\bar\Pi_2+q)^{\ast n}(x)
			&=\sum_{k=0}^n\left(n\atop k\right)\bar\Pi_1^{\ast k}\ast (\bar\Pi_2+q)^{\ast (n-k)}(x)\\
			&=\bar\Pi_1^{\ast n}(x)+\sum_{k=0}^{n-1}\left(n\atop k\right)\bar\Pi_1^{\ast k}\ast (\bar\Pi_2+q)^{\ast (n-k)}(x)
		\end{align*}
		the problem is reduced to pure convolutions of $\bar\Pi_1$ and mixed convolutions. The pure convolutions have been analyzed in the course of the proof of Theorem \ref{t4} and we have the claimed differentiability 			behavior of the theorem.\\
		The proof is complete if we can show that smoothness of $\bar\Pi_2+q$ makes the mixed convolutions $\bar\Pi_1^{\ast k}\ast (\bar\Pi_2+q)^{\ast (n-k)}$ everywhere infinitely differentiable away from $G_{k}$. To apply Lemma \ref{pre}, we use Lemma \ref{lem:dif} and (\ref{iter}) to see that $(\bar\Pi_2+q)^{\ast l}$ is everywhere infinitely differentiable and integrable at zero for arbitrary integer $l$. Hence, $\bar\Pi_1\ast (\bar\Pi_2+q)^{\ast l}$ is infinitely 				differentiable away from $G$ and furthermore integrable at zero by the same arguments used to derive (\ref{iter}). Inductively, Lemma \ref{pre} shows that $\bar\Pi_1^{\ast k}\ast (\bar\Pi_2+q)^{\ast l}$ is infinitely 				 differentiable away from $G_k$ for any integers $k,l$. As $k\leq n-1$ this shows that no additional points of non-differentiability are caused by $\bar\Pi_2$.
	\end{proof}

\subsection{Asymptotic Behavior}

	In contrast to the application to differentiability of the Laplace inversion representation, we now apply the series representation to find the asymptotics of $u^{(q)}$ and it's derivative at zero. The Laplace inversion representation is then applied to the asymptotics of ${u}'$ at infinity.
	
	\begin{proof}[Proof of Theorem \ref{p2}:]
		From Equation (\ref{series}) it follows that
		\begin{align}
			\frac{\frac{1}{\delta}-u^{(q)}(x)}{\frac{1}{\delta^2}\1\ast( q+\bar\Pi)(x)}=1+\sum_{n=2}^{\infty}\frac{(-1)^{n+1}\frac{\1\ast (q+\bar\Pi)^{\ast n}(x)}{\delta^{n+1}}}{\frac{1}{\delta^2}\1\ast(q+\bar\Pi)(x)}.\label{WE}
		\end{align}	
		Hence, we need to show that the latter summand of the right hand side converges to zero (absolutely) as $x$ tends to zero. To this end we use the estimate (\ref{iter}) to obtain
		\begin{align*}
			\left|\frac{\sum_{n=2}^{\infty}(-1)^{n+1}\frac{\1\ast (q+\bar\Pi)^{\ast n}(x)}{\delta^{n+1}}}{\frac{1}{\delta^2}(\1\ast (q+\bar\Pi)(x))}\right|\leq\sum_{n=1}^{\infty}\frac{\big(\1\ast(q+ \bar\Pi)(x)\big)^n}{\delta^{n}}.
		\end{align*}
		Clearly, $\lim_{x\to 0}\1\ast(q+\bar\Pi)(x)=0$ by the property (\ref{c}) so that for $x$ small enough the right hand side is bounded from above by
		\begin{align*}
			\frac{1}{1-\frac{\1\ast(q+\bar\Pi)(x)}{\delta}}-1.
		\end{align*}
		This shows that the right hand side vanishes for $x$ tending to zero proving (\ref{WE}). The higher order asymptotics follow in precisely the same manner.
	\end{proof}	
	The refined first order asymptotics can now be deduced directly:
	\begin{proof}[Proof of Corollary \ref{c1}:]
		By Theorem \ref{p2}, $u^{(q)}$ behaves asymptotically like $\frac{1}{\delta^2}(qx+\int_0^x\bar\Pi(y)\,dy)$ at zero. As $\bar\Pi$ decreases we have for $x$ sufficiently small
		\begin{align*}
			\frac{u^{(q)}(x)}{x}\geq\frac{ qx+x\bar\Pi(x)}{\delta^2 x}=\frac{q+\bar\Pi(x)}{\delta^2}
		\end{align*}
		and the right hand side diverges if the L\'evy measure is infinite. In the finite case we obtain similarly for $x$ sufficiently small
		\begin{align*}
			\frac{q+\bar\Pi(x)}{\delta^2}\leq \frac{u^{(q)}(x)}{x}\leq \frac{q+\bar\Pi(0)}{\delta^2}
		\end{align*}
		proving the claim.
	\end{proof}

	\begin{proof}[Proof of Theorem \ref{t7}:]
	 	Without loss of generality we assume $\delta=1$. Also consider the case when $\bar\Pi(0+)=\infty$. To prove the theorem, taking into account Equation (\ref{series}) we need to show that for $\beta(\Pi)<\frac{n}{n+1}$
		\begin{align}\label{hl}
			\lim_{x\to 0}\sum_{k=n+1}^{\infty}(-1)^k (q+\bar\Pi)^{\ast k}(x)=0.
		\end{align}
   		Note that taking into account Lemma \ref{l4} we may appeal to $q+\bar\Pi(y)\leq C y^{-(\beta(\Pi)+\eps)}$ for $\eps>0$ satisfying $\beta{(\Pi)}+\eps<1$.  Then for $y$ small enough
   		\begin{align*}
   			(q+\bar\Pi)^{*2}(y)&=\int_{0}^{y}(q+\bar\Pi)(z)(q+\bar\Pi(y-z))dz \\
   			&\leq C'\int_{0}^{y}z^{-(\beta(\Pi)+\eps)}(y-z)^{-(\beta(\Pi)+\eps)}dz\\
   			&=C_{2}y^{-2(\beta(\Pi)+\eps)+1}.
   		\end{align*}
   		By induction, for all $y$ such that $q+\bar\Pi(y)\leq y^{-\beta(\Pi)-\eps}$, it is easy to show that for each $k$ there is a constant $C_{k}$ depending only on $\beta(\Pi)$ such that
   		\begin{equation*}
   			(q+\bar\Pi)^{*k}(y)\leq C_{k}y^{-k(\beta(\Pi)+\eps)+k-1}.
   		\end{equation*}
   		For $\eps$ small enough, the quantity $-k(\beta(\Pi)+\eps)+(k-1)$ is strictly positive for $k>n$ by the assumption $\beta(\Pi)<\frac{n}{n+1}$. Therefore
   		$
   			\lim_{x\to 0}(q+\bar\Pi)^{*(n+1)}(x)=0
   		$
		and with (\ref{iter}) for $x>0$
   		\begin{align*}
   			(q+\bar\Pi)^{\ast(n+1+k)}(x)&\leq \sup_{y\leq x}\big((q+\bar\Pi)^{*(n+1)}(y)\big)\1*(q+\bar\Pi)^{*k}(x)\\
   			&\leq\sup_{y\leq x}\big((q+\bar\Pi)^{*(n+1)}(y)\big)(\1*(q+\bar\Pi)(x))^{k}.
   		\end{align*}
   		Again taking into account property (\ref{c}), we see that $\lim_{x\to 0}\1*(q+\bar\Pi)(x)=0$ and also $\lim_{x\to 0}(q+\bar\Pi)^{*(n+1)}(x)=0$. Furthermore, we have $q+\bar\Pi(x)\sim \bar\Pi(x)$ so that we can easily deduce (\ref{hl}).

   If $\bar\Pi(0+)<\infty$ then $\beta(\Pi)=0$ and clearly \eqref{hl} holds with $n=1$. Thus we conclude the proof.	
   \end{proof}

	For the asymptotic behavior at infinity we need to have a more carefull look at the Laplace inversion representation.
	
	\begin{proof}[Proof of Theorem \ref{t8}:]
		Formally setting $\lambda=0$ in the Laplace inversion integral of (\ref{0}) leads to
		\begin{align}
			{u}'(x+)&=-\frac{\bar\Pi(x+)}{\delta^2}+\sum_{n=2}^{N-1}\frac{(-1)^n}{\delta^{n+1}}\bar\Pi^{\ast n}(x)+\frac{1}{2\pi}\int_{-\infty}^{\infty}e^{i\theta x}m^N(\theta)\,d\theta,\\
			{u}'(x-)&=-\frac{\bar\Pi(x-)}{\delta^2}+\sum_{n=2}^{N-1}\frac{(-1)^n}{\delta^{n+1}}\bar\Pi^{\ast n}(x)+\frac{1}{2\pi}\int_{-\infty}^{\infty}e^{i\theta x}m^N(\theta)\,d\theta,
		\end{align}	
		where
		\begin{align*}
			m^N(\theta)=\frac{(-\mathcal F \bar\Pi(\theta))^N}{\delta^{N+1}(1+\frac{1}{\delta}\mathcal F \bar\Pi(\theta))}.
		\end{align*}
		To see that $m^N$ is well-defined, recall from Remark \ref{remark} that for finite mean L\'evy measure the denominator is bounded away from zero. To prove the representation rigorously we need to send $\lambda$ to zero in the Laplace inversion representation (\ref{0}), (\ref{00}). For this sake, convergence under the integral and the interchange of limit and integration need to be justified. Here, the finite mean assumption is crucial. As convergence of the exponential is trivial, we only show that $\lim_{\lambda\to 0}h^N(\lambda+i\theta)=h^N(i\theta)=m^N(\theta)$ for fixed $\theta$. By definition of $h^N$ it suffices to show
		\begin{align*}
			\lim_{\lambda\to 0}\mathcal L \bar\Pi^{\ast l}(\lambda+i\theta)=\mathcal L\bar\Pi^{\ast l}(i\theta)=\mathcal F\bar\Pi^{\ast l}(\theta)
		\end{align*}
		for $l=1$ and $l=N$. Here we use that the additional assumption $\E[X_1]<\infty$ is equivalent to $\int_0^{\infty}\bar\Pi(x)\,dx<\infty$ to apply dominated convergence justified by the uniform (in $\lambda$) upper bound
		\begin{align*}
			\int_0^{\infty}\Big|e^{-x(\lambda+i\theta)}\bar\Pi^{\ast l}(x)\Big|\,dx\leq \int_0^{\infty} \bar\Pi^{\ast l}(x)\,dx\leq \Big(\int_0^{\infty}\bar\Pi(x)\,dx\Big)^l<\infty.
		\end{align*}
		Having shown pointwise convergence of $h^N$, we now verify dominated convergence for $\lambda$ tending to zero. The estimate (\ref{he2}) is not strong enough as around the real axis the upper bound explodes with $\lambda$ tending to zero. The additional assumption again leads to the (coarse) uniform upper bound
		\begin{align*}
			&\big|Im(\mathcal L\bar\Pi(\lambda+i\theta))\big|=\left|\int_{0}^{\infty}\sin(\theta y)r_{\lambda}(y)dy\right|\leq \int_0^{\infty}\bar\Pi(x)\,dx<\infty
		\end{align*}
		and similarly for the real part $Re(\mathcal L\bar\Pi(\lambda+i\theta))$.  For large $\theta$ we employ the estimate (\ref{he2}) for sufficiently small $\eps>0$ so that we obtain from our choice of $N$ the integrable in $\theta$ upper bound (uniformly in $\lambda$)
		\begin{align}\label{xx}
			\left|h^N(\lambda+i\theta)\right|\leq  \begin{cases}
					\int_0^{\infty}\bar\Pi(x)\,dx&:|\theta|\leq 1,\\
       					C|\theta|^{N(\beta(\Pi)+\eps-1)}&:|\theta|>1.
			             \end{cases}		
		\end{align}
		Exploiting the Fourier inversion representation for $N$ large enough we can complete the proof. Apparently, it suffices to show that $\lim_{x\to \infty}\bar\Pi^{\ast n}(x)=0$ for $n=1,...,N-1$ and
		\begin{align*}
		 	\lim_{x\to \infty}\int_{-\infty}^{\infty}e^{i\theta x}m^N(\theta)\,d\theta=0.
		\end{align*}
		The first is a consequence of $\int_0^{\infty}\bar\Pi^{\ast n}(x)\,dx<\infty$ for which we take into account (\ref{iter}) and the assumption. The integral vanishes at infinity by the Riemann-Lebesgue theorem as soon as we justify that $m^N$ is integrable in $\theta$. But this follows from the upper bound (\ref{xx}) which, being uniform in $\lambda$, is also valid for $m^N(\theta)=h^N(i\theta)$.
		
	\end{proof}

\section*{Acknowledgment}
	The authors would like to thank Frank Aurzada for many discussions on the subject and an anonymous referee for a very careful reading of the manuscript.

\end{document}